\documentclass[11pt,DIV=11,a4paper,oneside,headings=small]{scrartcl}
\usepackage[T1]{fontenc}
\usepackage[english]{babel}
\usepackage{microtype}
\usepackage{amssymb,amsmath,amsfonts,amsthm,enumerate}
  \allowdisplaybreaks
\usepackage[noblocks]{authblk}

\newcommand{\elliptic}{\vartheta}
\newcommand{\euler}{\mathrm{e}}
\newcommand{\drm}{\mathrm{d}}
\newcommand{\RR}{\mathbb{R}}
\newcommand{\T}{\mathrm{T}}
\newcommand{\CC}{\mathbb{C}}
\newcommand{\NN}{\mathbb{N}}	
\newcommand{\ZZ}{\mathbb{Z}}
\renewcommand{\epsilon}{\varepsilon}
\newcommand{\weightbound}{\Xi}
\newcommand{\A}{\mathcal{A}}
\newcommand{\diag}{\operatorname{diag}}
\newcommand{\supp}{\operatorname{supp}}
\newcommand{\Tr}{\operatorname{Tr}}
\newcommand{\diver}{\operatorname{div}}
\renewcommand{\Re}{\operatorname{Re}}
\renewcommand{\Im}{\operatorname{Im}}
\newtheorem{theorem}{Theorem}[section]
\newtheorem{lemma}[theorem]{Lemma}
\newtheorem{proposition}[theorem]{Proposition}
\theoremstyle{definition}
\newtheorem*{assumption}{Assumption}
\theoremstyle{remark}
\newtheorem{remark}[theorem]{Remark}
\usepackage[textsize=tiny]{todonotes}
\usepackage[linktocpage, pdftex]{hyperref}
	\usepackage{color}
	\definecolor{darkred}{rgb}{0.5,0,0}
	\definecolor{darkgreen}{rgb}{0,0.5,0}
	\definecolor{darkblue}{rgb}{0,0,0.5}	\hypersetup{colorlinks,linkcolor=darkblue,filecolor=darkgreen,urlcolor=darkred,citecolor=darkblue}
	\hypersetup{
	pdftitle={A quantitative Carleman estimate for second order elliptic operators},
	pdfauthor={Ivica Naki\'c, Christian Rose and Martin Tautenhahn},
	}
\begin{document}
%
% %%%%%%%%%%%%%%%%%%%%%%%%%%%%%%%%%%%%%%%%%%%%%%%%%%%%%%%%%%%%%%%%%%%%%%
% %---------------------------------------------------------------------
% 
%         T I T L E
%
% %---------------------------------------------------------------------
% %%%%%%%%%%%%%%%%%%%%%%%%%%%%%%%%%%%%%%%%%%%%%%%%%%%%%%%%%%%%%%%%%%%%%%
% %
%
\title{A quantitative Carleman estimate for second order elliptic operators}
\author{Ivica Naki\'c}
\affil{University of Zagreb, Department of Mathematics, Croatia}
\author{Christian Rose}
\affil{Max Planck Institute for Mathematics in the Sciences, Leipzig,
Germany}

\author{Martin Tautenhahn}
\affil{Technische Universit\"at Chemnitz, Fakult\"at f\"ur Mathematik, Germany}
%
%\date{\today}
\date{\vspace{-5ex}}
\maketitle
\begin{abstract}
We prove a Carleman estimate for elliptic second order partial differential expressions with Lipschitz continuous coefficients. 
The Carleman estimate is valid for any complex-valued function $u\in W^{2,2}$ with support in a punctured ball of arbitrary radius.
The novelty of this Carleman estimate is that we establish an explicit dependence on the Lipschitz and ellipticity constants, the dimension of the space and the radius of the ball. In particular we provide a uniform and quantitative bound on the weight function for a class of elliptic operators given explicitly in terms of ellipticity and Lipschitz constant.
\end{abstract}
%
%\tableofcontents
%
%
% %
% %%%%%%%%%%%%%%%%%%%%%%%%%%%%%%%%%%%%%%%%%%%%%%%%%%%%%%%%%%%%%%%%%%%%%%
% %---------------------------------------------------------------------
% 
%         I N T R O D U C T I O N
%
% %---------------------------------------------------------------------
% %%%%%%%%%%%%%%%%%%%%%%%%%%%%%%%%%%%%%%%%%%%%%%%%%%%%%%%%%%%%%%%%%%%%%%
% %
%
%
\section{Introduction}
Carleman estimates were first introduced by Carleman \cite{Carleman-39} in 1939 in order to establish a unique continuation property for elliptic operators $L$ with non-analytic coefficients, i.e.\ if the solution $u$ of $L u = 0$ in $\Omega \subset \RR^d$
vanishes in a non-empty open set, then $u$ is identically zero. While Carleman's original result applies to the case $d =2$ and $L = -\Delta + V$ with $V \in L^\infty_{\mathrm{loc}} (\RR^2)$, by now there is a wealth of results pertaining to Carleman estimates
and its application to unique continuation, see, e.g., \cite{Hoermander-89}. In particular, notable attention has been paid to the case $V \in L^p_{\mathrm{loc}} (\RR^d)$, since this can be used to show that the Schr\"odinger operator $H = -\Delta + V$ with potential $V \in L^p_{\mathrm{loc}} (\RR^d)$ has no positive eigenvalues, see \cite{JerisonK-85} and the references therein. Further applications of Carleman estimates are, for example, uniqueness properties of solutions of Schr\"odinger equations \cite{KenigPV-02,IonescuK-06,EscauriazaKPV-11}, uniqueness and stability of inverse problems \cite{Klibanov-13,Klibanov-15}, or control theory of partial differential equations \cite{FursikovI-96,ColombiniZ-01,Rousseau-12,RousseauL-12}.

So far, all the mentioned applications of Carleman estimates concern qualitative statements. 
For its proofs, the particular dependence on the parameters is not essential. Certainly, there are plenty of applications which require some knowledge on the dependence on the parameters entering in the Carleman estimate. For this purpose, we cite the following Carleman estimate by Escauriaza and Vessella \cite{EscauriazaV-03}, see also \cite{KenigSU-11}.
Let $L = -\diver(A\nabla)$
be elliptic with Lipschitz continuous coefficients and denote the corresponding constants by $\vartheta_1$ and $\vartheta_2$, respectively. Then there exist $\kappa \in (0,1)$ and $C,\weightbound,\alpha_0 \geq 1$, depending on $\vartheta_1$ and $\vartheta_2$, and a function $w : \RR^d \to \RR_+$ satisfying $\lvert x \rvert / \weightbound \leq w (x) \leq \weightbound \lvert x \rvert$,
such that for all real-valued $u \in C_{\mathrm c}^\infty (B_\kappa \setminus \{0\})$ 
 and all $\alpha \geq \alpha_0$ we have
\begin{equation} \label{eq:EV03}
 \alpha \int w^{1-2\alpha} \lvert \nabla u \rvert^2 + \alpha^3 \int w^{-1-2\alpha}  u^2 \leq C \int w^{2-2\alpha} \bigl( Lu \bigr)^2 .
\end{equation}
Here, $B_\kappa$ denotes the open ball with radius $\kappa$ and center zero.
\par
The exceptional feature of this Carleman estimate is that the second term on the left hand side goes with $\alpha^3$ and the particular powers of the weight function. Thanks to this, Bourgain and Kenig \cite{BourgainK-05} were able to apply (a refined version of) this Carleman estimate to prove that if $\Delta u = V u$ in $\RR^d$, $u (0) = 1$, $\lvert u \rvert \leq C$ and $\lvert V \rvert \leq C$, then for all $x \in \RR^d$ with $\lvert x \rvert > 1$
 \begin{equation} \label{eq:BK05}
  \max_{\lvert y-x \rvert \leq 1} \lvert u (y) \rvert > c \cdot \exp \left( - c' (\log \lvert x \rvert) \lvert x \rvert^{4/3}\right) .
 \end{equation}
Let us note that the quantitative unique continuation principle in Ineq.~\eqref{eq:BK05} was crucial for the answer to a long-standing problem in the theory of random Schr\"odinger operators, namely
 Anderson localization for the continuum Anderson model with Bernoulli-distributed coupling constants. 
 Moreover, if the precise decay rate $4/3$ in Ineq.~\eqref{eq:BK05}, which results from the particular form of Ineq.~\eqref{eq:EV03}, would be replaced by 1.35, one could not conclude localization for the continuum Anderson-Bernoulli model using the same techniques, cf.\ \cite[p.~412]{BourgainK-05}.
 A $L^2$-variant of Ineq.~\eqref{eq:BK05} was also used in \cite{BourgainK-13} to prove bounds on the density of states measure for Schr\"odinger operators in dimension $d \in \{1,2,3\}$, a restriction which stems from the specific parameters in Ineq.~\eqref{eq:BK05}.
 \par
 Another example where the Carleman estimate \eqref{eq:EV03} proved to be useful is  so-called scale-free and quantitative unique continuation. In \cite{Rojas-MolinaV-13}, the authors prove that if $\lvert \Delta u \rvert \leq \lvert V u \rvert$ in $\Lambda_L = (-L/2 , L/2)^d$, then for all $\delta \in (0,1/2)$ one has
 \begin{equation} \label{eq:RMV13}
  \left( \frac{\delta}{C} \right)^{C + C \lVert V \rVert_{\infty}^{2/3}}\lVert u \rVert_{L^2 (\Lambda_L)}^2 \leq \lVert u \rVert_{L^2 (W_\delta (L))}^2 ,
 \end{equation}
where $C = C (d)$ and $W_\delta (L)$ is some union of equidistributed $\delta$-balls in $\Lambda_L$. The notion scale-free and quantitative refers to the fact that $C$ is independent of $L$ and the dependence on $\delta$ is known to be polynomial. The latter fact holds due to the bounds on the weight function given above and is essential for their application to random Schr\"odinger operators.
\par
The mentioned applications so far concerned Schr\"odinger operators only. In this case, the (abstract) constants $\kappa$ and $\Xi$ in Ineq.~\eqref{eq:EV03} depend only on the dimension. Due to the method of proof in \cite{BourgainK-05,BourgainK-13,Rojas-MolinaV-13}, a precise knowledge of $\kappa$ and $\Xi$ would change only dimension-dependent constants like $c$, $c'$ and $C$. However, if one, for example, pursues to prove an analogue of Ineq.~\eqref{eq:RMV13} for solutions of $\lvert Lu \rvert \leq \lvert V u \rvert$ with $L$ as above, one encounters dependencies of the constants $\kappa$ and $\Xi$ on the ellipticity and Lipschitz parameters of $L$. If we have no quantitative bounds on these parameters, the use of Carleman estimate \eqref{eq:EV03} might not be feasible, cf.\ Remark~\ref{rem:BTV} for details. Such problems have already been discussed in the recent proceedings \cite{BorisovTV-14} and \cite{BorisovNRTV-15} 
on 
quantitative scale-free unique continuation principles for elliptic operators.
This strongly suggests that a refinement of the Carleman estimate \eqref{eq:EV03} with explicit bounds on $\kappa$ and $\Xi$ is a worthwhile goal. 
\par
In this note we give a precise refinement of the Carleman estimate \eqref{eq:EV03} in the sense that
\begin{enumerate}[(i)]
	\item the estimate is valid on the whole punctured unit ball (i.e.\ $\kappa=1$), and
	\item all the constants, including the bound $\weightbound$ on the weight function, are explicitly calculated in terms of ellipticity and Lipschitz constant.
\end{enumerate}
Let us stress that our result has already been used to prove an analogous result to Ineq.~\eqref{eq:RMV13} for second order elliptic operators, see the recent paper \cite{BorisovTV} and Remark~\ref{rem:BTV}. In fact, the Carleman estimate from \cite{EscauriazaV-03} given in Ineq.~\eqref{eq:EV03} would not be sufficient to obtain such a result.
\par
Our proof is based on techniques developed in \cite{EscauriazaV-03,BourgainK-05}. 
Compared to the Carleman estimates of \cite{EscauriazaV-03} and \cite{BourgainK-05}, our result holds with $\kappa = 1$ and additional gradient term in the lower bound at the same time. Let us emphasize that, beside the quantitative control on all the parameters, even this involves non-trivial modifications of the existing proofs, cf.\ Section~\ref{sec:proof}.
\par
The paper is organized as follows. In Section~\ref{sec:main_result} we state the main result of the paper. In Section~\ref{sec:quantitative} we provide preparatory estimates and identities needed for the proof of the main result. In Section~\ref{sec:proof} we give the proof of the main result, while the proofs of two technical lemmata are postponed to the appendix.
%
%
% %
% %%%%%%%%%%%%%%%%%%%%%%%%%%%%%%%%%%%%%%%%%%%%%%%%%%%%%%%%%%%%%%%%%%%%%%
% %---------------------------------------------------------------------
% 
%         R E S U L T
%
% %---------------------------------------------------------------------
% %%%%%%%%%%%%%%%%%%%%%%%%%%%%%%%%%%%%%%%%%%%%%%%%%%%%%%%%%%%%%%%%%%%%%%
% %
%
%
\section{Main result}\label{sec:main_result}
Let $d \in \NN$ and $L$ be the second order partial differential expression
\begin{equation} \label{eq:operator}
 Lu := - \diver (A \nabla u) + b^\T \nabla u + c u = -\sum_{i,j=1}^d \partial_i \left( a^{ij} \partial_j u \right) + \sum_{i,j=1}^d b_i \partial_i u + c u ,
\end{equation}
acting on complex-valued functions on $\mathbb{R}^d$, where  $A : \mathbb{R}^d \to \mathbb{R}^{d \times d}$ with $A = (a^{ij})_{i,j=1}^d$, $b : \RR^d \to \CC^d$, $c : \RR^d \to \CC$, and $\partial_i$ denotes the $i$-th weak derivative.
Moreover, we denote by $B_\rho \subset \mathbb{R}^d$ the open ball in $\mathbb{R}^d$ with radius $\rho>0$ and center zero, by $\lvert z \rvert$ the Euclidean norm of $z \in \CC^d$, and by $\lVert M \rVert_\infty$, $\lVert M \rVert_1$ and $\lVert M \rVert$ the row sum, column sum and spectral norm of a matrix $M \in \CC^{d\times d}$. For the coefficient functions $A$, $b$ and $c$
we introduce the following assumption.
\begin{assumption}\label{ass:elliptic+}
Let $\rho > 0$, $\elliptic_1 \geq 1$ and $\elliptic_2 \geq 0$. We say that $\A(\rho,\elliptic_1 , \elliptic_2)$ is satisfied if and only if $ b, c \in L^\infty (B_\rho)$, $a^{ij} = a^{ji}$ for all $i,j \in \{1,\ldots , d\}$ and for all $x,y \in B_\rho$ and all $\xi \in \RR^d$ we have
\begin{equation} \label{eq:elliptic}
\elliptic_1^{-1} \lvert \xi \rvert^2 \leq \xi^\T A (x) \xi \leq \elliptic_1 \lvert \xi \rvert^2
\quad\text{and}\quad 
\lVert A(x) - A(y) \rVert_\infty \leq \elliptic_2 \lvert x-y \rvert .
\end{equation}
\end{assumption}
By Rademacher's theorem \cite{Federer-69}, if $\A (\rho , \elliptic_1 , \elliptic_2)$ is satisfied, then the coefficients $a^{ij}$ are differentiable almost everywhere on $B_\rho$ and the absolute value of the derivative is bounded by $\elliptic_2$. 
Operators $L$ for which the lower bound of the first inequality in \eqref{eq:elliptic} is satisfied are called elliptic. Elliptic means therefore, that for almost every point $x\in B_\rho$ the symmetric $d\times d$ matrix $A (x)$ is positive definite. A simple example is $A=I$ and $b=c=0$. In this case the operator $L$ coincides with the negative Laplacian. We use the notation $A_0 = A (0)$. 
\par
For $\mu , \rho > 0$ we introduce the function $w_{\rho , \mu} : \RR^d \to [0,\infty)$ by
\[
 w_{\rho , \mu} (x) := \varphi (\sigma (x / \rho)) ,
\]
where $\sigma:\mathbb{R}^d \to [0,\infty)$ and $\varphi : [0,\infty) \to [0,\infty)$ are given by
\begin{align*}
\sigma(x):= \left(x^\T A_0^{-1} x \right)^{1/2} \quad \text{and} \quad
\varphi(r):= r \exp \left( - \int_0^r\frac{1 - \euler^{-\mu t}}{t} \drm t \right).
\end{align*}
The function $\varphi$ obeys the upper bound $\varphi (r) \le r$. 
For a lower bound we distinguish two cases. If $\sqrt{\elliptic_1} \mu \leq 1$ we use the lower bound $\varphi (r) \geq r \euler^{-\mu r}$. If $\sqrt{\elliptic_1} \mu \geq 1$ distinguish two subcases. If $r \leq 1/\mu$ we use $1-\euler^{\mu t} \leq \mu t$ and obtain $\varphi (r) \geq r \euler^{-{\mu r}} \geq r (\euler \sqrt{\elliptic_1} \mu)^{-1}$. If $r > 1/\mu$ we split the integral according to the two components $[0,1/\mu]$ and $[1/\mu , r]$. On the first component we use $1-\euler^{-\mu t} \leq \mu t$, on the second one $1-\euler^{-\mu t} \leq 1$. This way we obtain $\varphi (r) \geq 1/(\euler \mu)$ if $r > 1/\mu$. Putting everything together we obtain for all $r \in [0,\sqrt{\elliptic_1}]$
\[
 \varphi (r) \geq \frac{r}{\mu_1} ,
 \quad \text{where} \quad 
 \mu_1 = \begin{cases}
          \euler^{\sqrt{\elliptic_1} \mu}& \text{if $\sqrt{\elliptic_1} \mu \leq 1$,}\\
          \euler \sqrt{\elliptic_1} \mu& \text{if $\sqrt{\elliptic_1} \mu \geq 1$}.
         \end{cases} 
\]
For the weight function $w_{\rho , \mu}$ there follows
\[
 \forall x \in B_\rho \colon \quad 
 \frac{\elliptic_1^{-1/2} \lvert x \rvert}{\rho \mu_1} \leq \frac{\sigma (x)}{\rho \mu_1}  \leq
 w_{\rho , \mu}(x) 
 \leq 
 \frac{\sigma (x)}{\rho} 
 \leq
 \frac{\sqrt{\elliptic_1} \lvert x \rvert}{\rho} .
\]
Keep in mind that we will drop the index of the weight function in Section~\ref{sec:quantitative} and onwards, and write $w$ instead of $w_{\rho , \mu}$.
Our main result is the following Carleman estimate.
\begin{theorem} \label{thm:carleman}
Let $\rho > 0$, $\elliptic_1 \geq 1$, $\elliptic_2 \geq 0$, Assumption $\A (\rho , \elliptic_1 , \elliptic_2)$ be satisfied and
\[
 \mu > 33 d \elliptic_1^{11/2} \elliptic_2 \rho .
\]
Then there exist constants 
$\alpha_0 = \alpha_0 (d , \rho, \elliptic_1 , \elliptic_2, \mu,\Vert b \Vert_\infty,\Vert c\Vert_\infty)$, and $C = C (d , \elliptic_1, \rho \elliptic_2, \mu) > 0$,
such that for all $\alpha \geq \alpha_0$ and all $u\in W^{2,2}(\RR^d)$ with support in $B_\rho \setminus \left\{ 0 \right\}$ we have
\begin{equation}
\label{eq:main_result}
 \int_{\mathbb{R}^d} \left( \alpha \rho^2 w_{\rho , \mu}^{1-2\alpha} 
 \nabla u^\T A \overline{\nabla u} 
 + 
 \alpha^3 w_{\rho , \mu}^{-1-2\alpha} \lvert u \rvert^2  \right) \drm x \leq 
 C \rho^4 \int_{\mathbb{R}^d} w_{\rho , \mu}^{2-2\alpha} \lvert Lu \rvert^2 \drm x .
\end{equation}
The constants $\alpha_0$ and $C$ are given in Eq.~\eqref{eq:alpha+C}.
\end{theorem}
\begin{remark}
In the case where $b$ and $c$ are identically zero, the conclusion of Theorem~\ref{thm:carleman} holds with $C = \tilde C$ and $\alpha_0 = \tilde \alpha_0$ where $\tilde C$ and $\tilde \alpha_0$ are given in Eq.~\eqref{eq:tildeC+alpha}. A straightforward calculation shows that they obey the upper bounds
\begin{align*}
 \tilde C 
 & \leq 2 d^2 \elliptic_1^8 \euler^{4\mu\sqrt{\elliptic_1}}   \mu_1^4 \left( 3 \mu^2 + (9\rho \elliptic_2 + 3)\mu + 1 \right) C_\mu^{-1} \\
 \intertext{and}
 \tilde \alpha_0& \leq  11 d^4 \elliptic_1^{33/2} \euler^{6\mu\sqrt{\elliptic_1}} \mu_1^6 (3\rho\elliptic_2 + \mu + 1)^2 \left(1  + \mu (\mu + 1) C_\mu^{-1}   \right)
\end{align*}
where $C_\mu = \mu - 33 d \elliptic_1^{11/2} \elliptic_2 \rho$. In the general case where $b,c \in L^\infty (B_\rho)$ the conclusion of the theorem holds with
\[
 C = 6 \tilde C 
 \quad \text{and} \quad
 \alpha_0 =\max \left\{ \tilde\alpha_0, C \rho^2  \lVert  b  \rVert_{\infty}^2 \elliptic_1^{3/2},
C^{1/3} \rho^{4/3} \lVert c \rVert_{\infty}^{2/3}\sqrt{\elliptic_1}   \right\} ,
\]
cf.\ Eq.~\eqref{eq:alpha+C}. We observe that if $\elliptic_2 \not = 0$, then both upper bounds of $C$ and $\alpha_0$ tend to infinity if $\mu$ tends to $33 d \elliptic_1^{11/2} \elliptic_2 \rho$. If $\elliptic_2 = 0$  the upper bound of $\tilde \alpha_0$ is uniformly bounded for all $\mu > 0$. 
In the case of the pure negative Laplacian, i.e.\ $A \equiv \operatorname{Id}$ and $b,c \equiv 0$, we have $\elliptic_1 = 1$ and $\elliptic_2 = 0$. In this case we note that our Carleman estimate is valid for arbitrary $\mu > 0$. For example, in the case $\mu = 1$ we infer from the proof of Theorem~\ref{thm:carleman} (not using the above estimates) that 
$C \leq 8 \euler^8  d^2$ and $\alpha_0 \leq 18 \euler^{12} d^4$.
\end{remark}
\begin{remark}\label{rem:BTV}
As mentioned in the introduction, the usefulness of our result has already been shown in the recent paper \cite{BorisovTV}. More precisely, based on the Carleman estimate from Theorem~\ref{thm:carleman}, they show so-called \emph{quantitative unique continuation principles, sampling theorems} as well as \emph{equidistribution theorems} for solutions of certain variable coefficient elliptic partial differential equations or inequalities.  In order to highlight the usefulness of our result, we formulate exemplary a special case of the sampling theorem from \cite{BorisovTV}. Thereafter, we discuss why, e.g., the Carleman estimate from Ineq.~\eqref{eq:EV03} is not satisfactory to obtain such a result.
Beforehand, we introduce some notation. 
\par
Let $\delta > 0$. We say that a sequence $z_j \in \RR^d$, $j \in \ZZ^d$, is \emph{$\delta$-equidistributed}, if
 \[
  \forall j \in \ZZ^d \colon \quad  B(z_j , \delta) \subset (-1/2 , 1/2)^d + j .
\]
Corresponding to a $\delta$-equidistributed sequence $z_j \in \RR^d$, $j \in  \ZZ^d$, we define the set
\[
S_{\delta} = \bigcup_{j \in \ZZ^d } B(z_j , \delta) \subset \RR^d .
\]
Note that the set $S_\delta$ depends on the choice of the $\delta$-equidistributed sequence.
\begin{theorem}[\cite{BorisovTV}]\label{thm:BTV}
Assume
\begin{equation} \label{ass:samplingG=1}
 \epsilon_1 := 1 - 33 \euler d (\sqrt{d} + 2) \elliptic_1^{6} \elliptic_2   > 0 .
\end{equation}
Then for all measurable and bounded $V : \RR^d \to \RR$, all $\psi \in W^{2,2} (\RR^d)$ and $\zeta \in L^2 (\RR^d)$ satisfying $\lvert L \psi \rvert \leqslant \lvert V\psi \rvert + \lvert \zeta \rvert$ almost everywhere on $\RR^d$, all $\delta \in (0,1/2)$ and all $\delta$-equi\-distri\-buted sequences we have
\[
 \lVert \psi \rVert_{S_\delta}^2 + \delta^2 \lVert \zeta \rVert_{\RR^d}^2 \geqslant c_{\mathrm{sfUC}} \lVert \psi \rVert_{\RR^d}^2 ,
\]
where
\begin{equation} \label{eq:definition-csfuc} 
 c_{\mathrm{sfUC}} = d_1 \left( \frac{\delta}{d_2} \right)^{\frac{d_3}{\epsilon_1} \bigl( 1 +  \lVert V \rVert_\infty^{2/3} + \lVert b \rVert_\infty^{2} + \lVert c \rVert_\infty^{2/3} \bigr) -\ln \epsilon_1}
\end{equation}
and $d_1$, $d_2$, and $d_3$ are constants depending only on $\elliptic_1$, $\elliptic_2$, and the dimension $d$. 
\end{theorem}
The constants $d_1$, $d_2$ and $d_3$ are given explicitly in \cite{BorisovTV}. 
Note that the exponent in \eqref{eq:definition-csfuc} can be interpreted 
as an estimate on the vanishing order as studied in \cite{DonnellyF-88,Kukavica-98}. The obtained estimate of vanishing order and its dependence on ellipticity and Lipschitz parameters originates from the particular form of our Carleman estimate, including $C$, $\alpha_0$, $w$, and the cubic behaviour of $\alpha$.
\par
Certainly, Theorem~\ref{thm:BTV} can be proven by using the Carleman estimate \eqref{eq:EV03} of \cite{EscauriazaV-03}. However, then Assumption \eqref{ass:samplingG=1} would involve some function of $\kappa$ and $\Xi$. Since $\kappa$ and $\Xi$ itself are abstract functions of $\elliptic_1$ and $\elliptic_2$, it would be not clear if the corresponding Assumption~\eqref{ass:samplingG=1} is ever satisfied. Moreover, in \cite{EscauriazaV-03} neither a monotonicity nor a continuity property of $\kappa$ and $\Xi$ is established. Thus even for operators with variable coefficients which are arbitrary close to the Laplacian, it is not possible to use solely the Carleman estimate of \cite{EscauriazaV-03} to ensure an estimate like \eqref{ass:samplingG=1}. For this reason, it is even not possible to verify the corresponding Assumption \eqref{ass:samplingG=1} by choosing $\elliptic_1$ and $\elliptic_2$ 
sufficiently small. 
\end{remark}
\begin{remark}
  Our result gives rise to a uniform Carleman estimate for a certain class of elliptic operators given precisely in terms of ellipticity and Lipschitz constants. More precisely, fix $\hat\elliptic_1\geq1$ and $\hat\elliptic_2 , \hat b , \hat c > 0$, $\rho > 0$ and denote by $\mathcal L$ the class of elliptic operators of the form \eqref{eq:operator} for which
\begin{enumerate}[(i)]
 \item the ellipticity and Lipschitz constant on $B_\rho$ are bounded by $\hat \elliptic_1$ and $\hat \elliptic_2$, and
 \item the coefficient functions $b$ and $c$ satisfy $\lVert  b  \rVert_\infty \leq \hat b$ and $\lVert c \rVert_\infty \leq \hat c$.
\end{enumerate}
Then, by monotonicity we can apply our result uniformly for all $L \in \mathcal L$ with, e.g., $\mu = 34 d \hat\elliptic_1^{11/2} \hat\elliptic_2 \rho$. In particular, there are constants $\tilde\alpha_0$ and $\tilde C$, depending only on $d$, $\hat \elliptic_1$, $\hat \elliptic_2$, $\hat b$, $\hat c$ and $\rho$, and a function $w = w_{\rho , \mu}:\RR^d \to [0,\infty)$, 
such that \eqref{eq:main_result} holds 
for all elliptic operators in $L \in \mathcal L$.
 \end{remark}
%
%
%
%
%
%
%
% %
% %%%%%%%%%%%%%%%%%%%%%%%%%%%%%%%%%%%%%%%%%%%%%%%%%%%%%%%%%%%%%%%%%%%%%%
% %---------------------------------------------------------------------
% 
%         Q U A N T I T A T I V E    E S T I M A T E S 
%
% %---------------------------------------------------------------------
% %%%%%%%%%%%%%%%%%%%%%%%%%%%%%%%%%%%%%%%%%%%%%%%%%%%%%%%%%%%%%%%%%%%%%%
% %
%
%
\section{Preliminary relations and quantitative estimates} \label{sec:quantitative}
In this section we provide some preparatory estimates and identities for the proof of Theorem~\ref{thm:carleman}. In particular, we will prove the quantitative estimates needed for our quantitative version of a Carleman estimate. First we introduce some notation. Let $L_0$ be the part of $L$ containing only the second order term, i.e.
\[
 L_0 u := - \diver (A \nabla u) .
\]
For $k \in \{1,\ldots , d\}$ let $e_k \in \mathbb{R}^d$ be the $k$-th unit vector. For $g \in \{\sigma , w\}$ and real-valued $f \in C_{\mathrm{c}}^\infty(\mathbb{R}^d)$ we define the functions $F_g^A: \mathbb{R}^d \to \mathbb{R}$, $D_f : \mathbb{R}^d \to \mathbb{R}$ and $h_g^A : \mathbb{R}^d \to \mathbb{R}^d$ given by
\[
 F_g^A:= -\frac{gL_0 g}{\nabla g^\T A \nabla g} - 1 ,
 \quad 
 D_f := w \frac{\nabla f^\T A \nabla w}{\nabla w^\T A\nabla w} + \frac{1}{2} f F_w^A ,
 \quad \text{and} \quad
 h_g^A := \frac{g A \nabla g}{\nabla g^\T A \nabla g} .
\]
Note that $F_g^{A_0}$ refers to the function $F_g^{A_0} = -g\diver(A_0 \nabla g) /(\nabla g^\T A_0 \nabla g) - 1$. Moreover, we introduce the matrix-valued function $M_g^A :\mathbb{R}^d \to \mathbb{R}^{d \times d}$,
\[
 M_g^A: = -\frac{1}{2} F_g^A A + \frac{1}{2} \diver (h_g^A \circ A) - \frac{1}{2} \left( A D(h_g^A) + \left[A D(h_g^A)\right]^\T \right) .
\]
Here $h_g^A \circ A$ denotes entrywise multiplication, i.e.\ $h_g^A \circ A: = (h_g^A a^{ij})_{i,j=1}^d$, while the divergence of a matrix denotes entrywise divergence. The matrix $D(h_g^A)$ is defined as the $d \times d$ matrix $(\nabla (h_g^A)_1 \ \nabla (h_g^A)_2  \ldots \nabla (h_g^A)_d)$, i.e.\ $D(h_g^A) := (\partial_i e_j^\T h_g^A)_{i,j=1}^d$. Finally, we define $\psi : (0,\infty) \to (0,\infty)$,  $\psi(r) := \varphi(r) /( r\varphi'(r))$. Note that
$\psi (r) = \mathrm{e}^{\mu r}$.

The following lemma provides some basic properties of the functions introduced above.
\begin{lemma} \label{lemma:basic_properties}
Let $\rho = 1$, $\elliptic_1 \geq 1$, $\elliptic_2 \geq 0$ and $\A (1 , \elliptic_1 , \elliptic_2)$ be satisfied. Then, almost everywhere on $B_1$, we have for the matrix-valued functions $M_w^A$ and $M_\sigma^A$ the relations
\[
 M_w^A = \psi (\sigma) M_\sigma^A + \sigma \psi' (\sigma) \left[ A - A \frac{\nabla \sigma \nabla\sigma^\T}{\nabla\sigma^T A \nabla \sigma} A \right] , 
 \quad
 M_\sigma^{A_0} = 0
 \quad \text{and} \quad
 M_\sigma^A \nabla \sigma = 0,
\]
and for the functions $F_w^A$, $F_\sigma^A$ and $F_\sigma^{A_0}$ the relations
\[
 F_w^A = \psi (\sigma) F_\sigma^A - \sigma \psi' (\sigma)
 \quad \text{and} \quad F_\sigma^{A_0} = d-2 .
\]
\end{lemma}
The relations of Lemma~\ref{lemma:basic_properties} can be found in \cite{EscauriazaV-03}, see also \cite{MorassiRV-11}. Note that our matrix-valued function $M_g^A$ coincides with the matrix-valued function $S_g^A$ of \cite{MorassiRV-11}, $g \in \{w,\sigma\}$. Indeed, the matrix $S_g^A$ is defined by
\begin{equation*}
2 (S_g^A)_{ij}
:= 
- F_g^A a^{ij} \delta_{ij}
+ \diver (h_g^A) a^{ij} \delta_{ij} 
- \partial_k (h_g^A)_j a^{ki} 
- \partial_k (h_g^A)_i a^{kj} 
+  (h_g^A)_k \partial_k a^{ij} ,
\end{equation*}
which coincides with $2 (M_g^A)_{ij}$. However, since no details of the proof are given in \cite{EscauriazaV-03,MorassiRV-11} we give a proof of Lemma~\ref{lemma:basic_properties} in Appendix~\ref{app:lemma_2.1}.
\par
In the following proposition we prove the quantitative estimates needed for our quantitative version of a Carleman estimate. 
\begin{proposition} \label{prop:quantitative}
Let $\rho = 1$, $\elliptic_1 \geq 1$, $\elliptic_2 \geq 0$ and $\A (1 , \elliptic_1 , \elliptic_2)$ be satisfied. Then, for all $\xi \in \mathbb{R}^d$ and almost everywhere on $B_1$ we have the estimates
 \[
  \lvert F_\sigma^A - F_\sigma^{A_0} \rvert \leq \sigma C'_F ,
  \quad
  \lvert \xi^\T M_\sigma^A \xi \rvert \leq \sigma C_M \xi^\T A \xi ,
  \quad 
  \lvert F_w^A \rvert \leq C_F ,
  \quad \text{and} \quad
  \lvert L_0 \psi (\sigma) \rvert \leq C_\psi / \sigma ,
 \]
where
 \begin{align*}
C'_F   & :=3 d \elliptic_1^{7/2} \elliptic_2,&
C_F    & :=\euler^{\mu \sqrt{\elliptic_1}} \left( \sqrt{\elliptic_1} (C'_F + \mu) + \lvert d-2 \rvert  \right),  \\
C_M    & :=11  d \elliptic_1^{11/2}  \elliptic_2,  &
C_\psi & := \mu \euler^{\mu \sqrt{\elliptic_1}} \elliptic_1^2
 \left( \sqrt{\elliptic_1} (C'_F + \mu) + d-1  \right) .
\end{align*}
\end{proposition}
\begin{proof}
Recall that the coefficients $a^{ij}$ are almost everywhere on $B_1$ differentiable.
We start by estimating $\lvert F_\sigma^A - F_\sigma^{A_0} \rvert$. By Lemma~\ref{lemma:basic_properties} and the definition of $F_\sigma^A$ we have
\[
 F_\sigma^A - F_\sigma^{A_0} = -\frac{\sigma L_0 \sigma}{\nabla\sigma^\T A \nabla\sigma} - 1 - (d-2) .
\]
We calculate 
\begin{equation*}
 \nabla\sigma^\T A \nabla\sigma = \sigma^{-2} x^\T A_0^{-1} A A_0^{-1} x
 \quad \text{and} \quad
 L_0 \sigma = \sigma^{-3} x^\T A_0^{-1} A A_0^{-1} x - \sigma^{-1} \diver \left( A A_0^{-1} x \right)
\end{equation*}
and obtain
\begin{equation} \label{eq:BsigmaA}
 F_\sigma^A - F_\sigma^{A_0}
 = \frac{\sigma^2 \diver \left( A A_0^{-1} x \right)}{x^\T A_0^{-1} A A_0^{-1} x}  - d  
 = \frac{\sigma^2  \sum_{i,l} \left(\partial_i a^{il} \right) e_l^\T A_0^{-1} x  + \sigma^2 \Tr \left( A A_0^{-1} \right) }{x^\T A_0^{-1} A A_0^{-1} x}  - d .
\end{equation}
By our assumption on ellipticity and Lipschitz continuity we have $\lvert e_l^\T A_0^{-1} x \rvert \leq \sqrt{\elliptic_1} \sigma$, $\sum_{i,l=1}^d \lvert \partial_i a^{il} \rvert \leq d \elliptic_2$ and $x^\T A_0^{-1} A A_0^{-1} x \geq \elliptic_1^{-2} \sigma^2$. Hence,
\[
 \lvert F_\sigma^A - F_\sigma^{A_0} \rvert
 \leq  \sigma d \elliptic_1^{5/2}  \elliptic_2 + \biggl\lvert \frac{\sigma^2 \Tr \left( A A_0^{-1} \right)}{x^\T A_0^{-1} A A_0^{-1} x}  - d \biggr\rvert = \sigma d \elliptic_1^{5/2}  \elliptic_2  + \lvert \Tr (T) \rvert,
\]
where
\[
 T := g A A_0^{-1} - \operatorname{Id} \quad \text{and} \quad  g := \frac{\sigma^2}{x^\T A_0^{-1} A A_0^{-1} x}.
\]
Next we estimate $\lvert \Tr (T) \rvert \leq d\lVert T \rVert$. We have
\begin{align*}
 \lVert T \rVert &= \lVert g A A_0^{-1} - g A_0 A_0^{-1} + g A_0 A_0^{-1} - \operatorname{Id} \rVert 
 \leq \lvert g \rvert \lVert A A_0^{-1} - A_0 A_0^{-1} \rVert + \lvert g - 1  \rvert .
\end{align*}
By 

\medskip
By our ellipticity assumption we have $g = \sigma^2 / ( x^\T A_0^{-1} A \allowbreak A_0^{-1} x ) \leq \elliptic_1^2$ and
\begin{align*}
\lvert g-1  \rvert
= \left\lvert \frac{x^\T A_0^{-1}(A_0 -A) A_0^{-1}x}{x^\T A_0^{-1} A A_0^{-1} x} \right\rvert 
&\le \lVert A-A_0 \rVert \frac{x^\T A_0^{-1} A_0^{-1}x}{x^\T A_0^{-1} A A_0^{-1} x}  \le \elliptic_1 \lVert A-A_0 \rVert .
\end{align*}
For the norm $\lVert A_0 - A \rVert$ we use our assumption on Lipschitz continuity and obtain
\begin{equation} \label{eq:Lipschitz}
  \lVert A_0 - A \rVert \leq \sqrt{\lVert A-A_0 \rVert_1 \lVert A - A_0 \rVert_\infty} \leq \elliptic_2 \lvert x \rvert \leq \elliptic_2 \sqrt{\elliptic_1} \sigma.
\end{equation}
By the same argument we have $\lVert A A_0^{-1} - A_0 A_0^{-1} \rVert \leq \elliptic_1^{3/2}  \elliptic_2 \sigma$. Hence we have $\lVert T \rVert \leq 2\elliptic_1^{7/2} \elliptic_2 \sigma$ and we obtain the claimed bound $\lvert F_\sigma^A - F_\sigma^{A_0} \rvert \leq \sigma C'_F$. 
\par
Now we turn to the estimate $\lvert \xi^\T M_\sigma^A \xi \rvert \leq \sigma C_M \xi^\T A \xi$. From Lemma~\ref{lemma:basic_properties} we infer that $M_\sigma^{A_0} A_0^{-1} A = 0$. Note also that $h_\sigma^{A_0}  = x$ and $D (h_\sigma^{A_0}) = I$. Hence
  \begin{equation*} \label{eq:M}
  M_\sigma^A = M_\sigma^A - M_\sigma^{A_0} A_0^{-1} A  = S_1 + S_2 + S_3 + S_3^\T
  \end{equation*}
  where
  \begin{align*}
  S_1 := -\frac 12 ( F_\sigma^A  -  F_\sigma^{A_0} )A,\!\!\quad
  S_2 :=\frac{1}{2}\left( \diver (h_\sigma^{A} \circ A) -  d A \right)  
  \quad\!\! \text{and} \!\!\quad
  S_3 := - \frac{1}{2} A \left( D(h_\sigma^A) -  I \right)  .
  \end{align*}
For the second summand $S_2$ we have by the product rule for the divergence and Eq.~\eqref{eq:BsigmaA} 
\begin{equation*}
 S_2 = \frac{1}{2} \left( \left[\nabla \left( \frac{\sigma^2 a^{ij}}{x^\T A_0^{-1} A A_0^{-1} x} \right)^\T A A_0^{-1} x\right]_{i,j=1}^d + A (F_\sigma^A - F_\sigma^{A_0}) \right) =: S_5 + S_6 .
\end{equation*}
Hence, $S_1$ and $S_6$ cancel out and we have $M_\sigma^A = S_5 + S_3 + S_3^\T$. For $S_5$ we calculate
\[
S_5 = \frac{1}{2} S_7 + \frac{1}{2} S_8  + S_9 ,
\]
where
\begin{align*}
S_7 &:= \frac{\sigma^2}{x^\T A_0^{-1} A A_0^{-1} x} \left[ (\nabla a^{ij})^\T A A_0^{-1} x \right]_{i,j=1}^d, \\
S_8 &:= - \frac{\sigma^2 A}{(x^\T A_0^{-1} A A_0^{-1} x)^2} \sum_k  x^\T A_0^{-1} (\partial_k A) A_0^{-1} x  e_k^\T A A_0^{-1} x, \\
S_9 &:= A   -  \frac{\sigma^2 A}{(x^\T A_0^{-1} A A_0^{-1} x)^2}  
x^\T A_0^{-1} A A_0^{-1} A A_0^{-1} x .
\end{align*}
Here $\partial_k A$ denotes the matrix $(\partial_k a^{ij})_{i,j=1}^d$. For the Frobenius norm of $S_7$ we have using $g \leq \elliptic_1^2$, $\sum_{j} \lvert \partial_k a^{ij}\rvert^2 \leq \elliptic_2^2$ and $\lvert A A_0^{-1} x \rvert \leq \elliptic_1^{3/2} \sigma$
\begin{equation*}
\lVert S_7 \rVert_{\mathrm F} 
= g \sqrt{ \sum_{i,j=1}^d \left\lvert  (\nabla a^{ij})^\T A A_0^{-1} x  \right\rvert^2 } \leq g \sqrt{\sum_{i,j,k=1}^d \lvert \partial_k a^{ij} \rvert^2 \lvert (A A_0^{-1} x)_k \rvert^2 }
\leq \sqrt{d} \elliptic_1^{7/2} \elliptic_2 \sigma ,
\end{equation*}
and hence we obtain
\[
\lvert \xi^\T S_7  \xi \rvert \leq \lVert S_7 \rVert \xi^\T \xi \leq  \lVert S_7 \rVert_{\mathrm F} \xi^\T  \xi \leq   \lVert S_7 \rVert_{\mathrm F} \elliptic_1 \xi^\T A  \xi 
\leq  \sqrt{d} \elliptic_1^{9/2} \elliptic_2 \sigma \xi^\T A \xi . 
\]
For $S_8$ we use H\"older's inequality and $\lvert A A_0^{-1} x \rvert \leq \sqrt{\elliptic_1} x^\T A_0^{-1} A  A _0^{-1}x$\ and obtain
\begin{align*}
 \lvert \xi^\T S_8 \xi \rvert 
 \leq \frac{\xi^\T A \xi \sigma^2 \sqrt{\elliptic_1}}{(x^\T A_0^{-1} A A_0^{-1} x)^{3/2}} \left(\sum_{k=1}^d \lvert x^\T A_0^{-1} (\partial_k A) A_0^{-1} x \rvert^2 \right)^{1/2}
\end{align*}
Since $\lVert \partial_k A \rVert \leq \sqrt{d}\lVert \partial_k A \rVert_{\infty} \leq \sqrt{d} \elliptic_2$ and $x^\T A_0^{-1} A A_0^{-1} x \geq \elliptic_1^{-1} x^\T A_0^{-1} A_0^{-1} x$ we obtain
\begin{align*}
 \lvert \xi^\T S_8 \xi \rvert 
 \leq \frac{\xi^\T A \xi \sigma^2 \sqrt{\elliptic_1}}{(x^\T A_0^{-1} A A_0^{-1} x)^{3/2}} d\elliptic_2  x^\T A_0^{-1} A_0^{-1} x 
 \leq
 \frac{\xi^\T A \xi \sigma^2  \elliptic^{4/2}}{(x^\T A_0^{-1}  A_0^{-1} x)^{1/2}} d\elliptic_2  
 \leq
d \elliptic_1^{5/2} \elliptic_2 \sigma \xi^\T A \xi .
\end{align*}
For $S_9$ we calculate
\begin{align*}
\lvert \xi^\T S_9 \xi \rvert 
&=  \frac{\xi^\T A \xi}{(x^\T A_0^{-1} A A_0^{-1} x)^2}
\left\lvert (C^\T x)^\T (y-x) - (y-x)^\T Cx \right\rvert 
 \leq  \frac{ \xi^\T A \xi (\lvert C x \rvert + \lvert C^\T x \rvert)  \lvert x-y \rvert }{(x^\T A_0^{-1} A A_0^{-1} x)^2}  ,
\end{align*}
where we used the notations $C = A_0^{-1} A A_0^{-1} x x^\T A_0^{-1}$ and $y = A A_0^{-1} x$.
The inequalities
$\lvert C x \rvert \leq \sigma^2 \elliptic_1^{3/2} (x^\T A_0^{-1} A A_0^{-1} x)^{1/2}$, $\lvert C^\T x \rvert \leq \sigma^2 \elliptic_1^{3/2} (x^\T A_0^{-1} A A_0^{-1} x)^{1/2}$, $\lvert x-y \rvert  \leq \lVert A - A_0 \rVert (x^\T A_0^{-1} A_0^{-1} x)^{1/2}$, and $x^\T A_0^{-1} A A_0^{-1} x \geq \elliptic_1^{-1} x^\T A_0^{-1}  A_0^{-1} x$ imply
\[
\lvert \xi^\T S_9 \xi \rvert 
\leq
\frac{2 \xi^\T A \xi \sigma^2 \elliptic_1^{3/2} \lVert A - A_0 \rVert (x^\T A_0^{-1} A_0^{-1} x)^{1/2} }{(x^\T A_0^{-1} A A_0^{-1} x)^{3/2}}
\leq
\frac{2 \xi^\T A \xi \sigma^2 \elliptic_1^{3} \lVert A - A_0 \rVert }{x^\T A_0^{-1}  A_0^{-1} x} .
\]
Using $x^\T A_0^{-1} A_0^{-1} x \geq \elliptic_1^{-1} \sigma^2$ and Ineq.~\eqref{eq:Lipschitz} we obtain
\[
 \lvert \xi^\T S_9 \xi \rvert  \leq
 2 \xi^\T A \xi \elliptic_1^{4} \lVert A - A_0 \rVert \leq 2 \xi^\T A \xi \elliptic_1^{9/2} \elliptic_2 \sigma .
\]
For $S_5$ we finally obtain the estimate
\begin{equation} \label{eq:S2}
\lvert \xi^\T S_5 \xi \rvert \leq 3 d \elliptic_1^{9/2} \elliptic_2 \sigma \xi^\T A \xi .
\end{equation}
Now we start estimating $S_3$. We have $D (h_\sigma^A) = g D (A A_0^{-1} x)+ \operatorname{diag} (\nabla g) X A_0^{-1} A$, where $X := (x \cdots x)^\T$ and $\diag (v) := \diag (v_1, \ldots , v_d)$ for $v \in \mathbb{R}^d$. For $\partial_k g$ we calculate
\begin{equation*}
\partial_k g
=
\frac{2 e_k^\T A_0^{-1} x}{x^\T A_0^{-1} A A_0^{-1} x} 
-
\frac{x^\T A_0^{-1} x}{(x^\T A_0^{-1} A A_0^{-1} x)^2}
\left(
2 e_k^\T A_0^{-1} A A_0^{-1} x + x^\T A_0^{-1} (\partial_k A) A_0^{-1} x
\right) .
\end{equation*}
Hence,
\[
S_3 = -\frac{1}{2} A \left( S_{10} + S_{11} + S_{12} \right) 
\]
where
\begin{align*}
S_{10} &:= g D (A A_0^{-1} x) - I , \\
S_{11} &:= \frac{2}{x^\T A_0^{-1} A A_0^{-1} x} 
\left[ A_0^{-1} x - g  A_0^{-1} A A_0^{-1} x \right]
x^\T A_0^{-1} A , \\
S_{12} &:= -\frac{g}{x^\T A_0^{-1} A A_0^{-1} x}  (x^\T A_0^{-1} (\partial_i A) A_0^{-1} x)_{i=1}^d x^\T A_0^{-1} A .
\end{align*}
For the norm of $S_{10}$ we calculate using $D (x) = I$
\begin{align*}
\lVert S_{10} \rVert 
% &= \lVert g D (A A_0^{-1} x) - g D(x) + g D(x) - D (x) \rVert \\
& \leq g \lVert D (A A_0^{-1} x-x) \rVert + \lvert g-1 \rvert 
=g \left\lVert A_0^{-1}A  + C - I \right\rVert + \lvert g-1 \rvert \\
& \leq g \lVert A_0^{-1} \rVert \cdot \lVert A - A_0 \rVert + g \left\lVert C \right\rVert
+ \lvert g-1 \rvert ,
\end{align*}
where 
\[
C = \Bigl(\sum_{k=1}^d (\partial_i a^{jk}) e_k^\T A_0^{-1} x \Bigr)_{i,j=1}^d .
\]
The existing bounds $g\leq \elliptic_1^2$, $\lvert g-1 \rvert \leq \elliptic_1^3 \lVert A - A_0 \rVert$, $\lVert A-A_0 \rVert \leq \elliptic_2 \sqrt{\elliptic_1} \sigma$ and
\[
\lVert C \rVert \leq  \lVert C \rVert_{\mathrm F} 
\leq  \sqrt{\sum_{i,j,k=1}^d  \lvert \partial_i a^{jk} \rvert^2 \lvert e_k^\T A_0^{-1} x \rvert^2 } 
\leq  \sqrt{\sum_{i=1}^d \elliptic_2^2 \lvert A_0^{-1} x \rvert^2 }
\leq \sqrt{d \elliptic_1} \elliptic_2 \sigma
\]
give us $\lVert S_{10} \rVert \leq 3 \sqrt{d} \elliptic_1^{7/2}\elliptic_2 \sigma$.
For $S_{11}$ we calculate using Cauchy-Schwarz's  inequality, $\lvert A A_0^{-1} x \rvert \leq \sqrt{\elliptic_1} (x^\T A_0^{-1} A A_0^{-1} x)^{1/2}$, $x^\T A_0^{-1} A A_0^{-1} x \geq \elliptic_1^{-1} x^\T A_0^{-1} A_0^{-1} x$, $T = g A A_0^{-1} - I$, and $\lVert T \rVert \leq 2 \elliptic_1^{7/2} \elliptic_2 \sigma$ as before 
\[
 \lVert S_{11} \rVert \leq \frac{2\sqrt{\elliptic_1}}{(x^\T A_0^{-1} A A_0^{-1} x)^{1/2}} \lvert (I - g A_0^{-1} A) A_0^{-1} x \rvert \leq 2 \elliptic_1 \lVert T \rVert \leq 4 \elliptic_1^{9/2} \elliptic_2 \sigma . 
\]
For $S_{12}$ we use Cauchy Schwarz's inequality and $\lvert A A_0^{-1} x \rvert \leq \sqrt{\elliptic_1} (x^\T A_0^{-1} A A_0^{-1} x)^{1/2}$, and obtain
\begin{equation*}
 \lVert S_{12} \rVert 
  \leq \frac{\sqrt{\elliptic_1} g}{(x^\T A_0^{-1} A A_0^{-1} x)^{1/2}}  \lvert (x^\T A_0^{-1} (\partial_i A) A_0^{-1} x)_{i=1}^d \rvert 
\end{equation*}
Since $\lvert x^\T A_0^{-1} (\partial_i A) A_0^{-1} x \rvert \leq \sqrt{d} \elliptic_2 x^\T A_0^{-1} A_0^{-1} x$ we have $\lvert (x^\T A_0^{-1} (\partial_i A) A_0^{-1} x)_{i=1}^d \rvert \leq d \elliptic_2 x^\T \allowbreak A_0^{-1} \allowbreak A_0^{-1} x$. Hence, using $x^\T A_0^{-1} A A_0^{-1} x \geq \elliptic_1^{-1} x^\T A_0^{-1} A_0^{-1} x$, $g \leq \elliptic_1^2$ and $x^\T A_0^{-1} A_0^{-1} x \leq \elliptic_1 \sigma^2$ we obtain
\[
  \lVert S_{12} \rVert \leq \elliptic_1 g d \elliptic_2 (x^\T A_0^{-1} A_0^{-1} x)^{1/2} 
  \leq d \elliptic_1^{7/2}  \elliptic_2 \sigma .
\]
Note that for any matrix $C$ we have by ellipticity of $A$ the estimate $\xi^\T A C \xi \leq \elliptic_1 \lVert C \rVert \xi^\T A \xi$.
% \todo[inline, size=\small]{At the moment, I can obtain the bound with $\elliptic_1^2$. I remember that we started with such a bound but then we improved it. I recall how we obtained $\elliptic_1^2$, but not how we obtained $\elliptic_1$.

% Proof that I recall:
% \[
% x^\T AC x = (ACx,x) = (Cx,x)_A \le \lVert C \rVert_A (x,x)_A = \lVert C \rVert_A x^\T A x, 
% \]
% \[
% \lVert C \rVert_A = \sup_x \frac{\lVert Cx \rVert_A }{\lVert x \rVert_A } = \sup_x \frac{\lVert ACx \rVert }{\lVert Ax \rVert } \le \elliptic_1^2 \sup_x \frac{\lVert Cx \rVert }{\lVert x \rVert } = \elliptic_1^2 \lVert C \rVert. 
% \]
% I think $\lVert Cx \rVert_A \not = \lVert ACx \rVert$! Instead we have $\lVert Cx \rVert_A = \sqrt{(ACx , Cx)} \leq \sqrt{\elliptic_1} \sqrt(Cx,Cx) = \sqrt{\elliptic_1} \lVert Cx \rVert$.
% \\[2ex]
% \textbf{Another try:} Let $(x,y)_A = (Ax , y)$. Then $\xi^\T AC \xi = (\xi , AC \xi) = (\xi , C \xi)_A$ and we have
% \[
%  \xi^T AC \xi = (\xi , C\xi)_A \leq \lVert C \rVert_A (\xi , \xi)_A .
% \]
% For $\lVert C \rVert_A$ we have
% \[
%  \lVert C \rVert_A 
%  =    \sup_x \frac{\lVert Cx \rVert_A }{\lVert x \rVert_A } 
%  =    \sup_x \frac{\sqrt{(Cx , Cx)_A}}{\sqrt{(x,x)_A}}
%  \leq \sup_x \frac{\sqrt{\elliptic_1} \sqrt{(Cx,Cx)}}{\elliptic_1^{-1/2} \sqrt{(x,x)}} = \elliptic_1 \lVert C \rVert .
% \]
% Hence,
% \[
%  \xi^T AC \xi \leq \elliptic_1 \lVert C \rVert (\xi , \xi)_A
%  = \elliptic_1 \lVert C \rVert \xi^\T A \xi .
% \]

% }
Hence, we obtain
\begin{equation} \label{eq:S3}
\lvert \xi^\T S_3 \xi \rvert
\leq 
\frac{1}{2} \xi^\T A \xi \elliptic_1 \lVert S_{10} + S_{11} + S_{12} \rVert \leq  4 d \elliptic_1^{11/2} \elliptic_2 \sigma \xi^\T A \xi .
\end{equation}
From Ineq.~\eqref{eq:S2} and \eqref{eq:S3} we conclude that $\lvert \xi^\T M_\sigma^A \xi \rvert = \lvert \xi^\T (S_5 + S_3 + S_3^\T) \xi \rvert \leq \sigma C_M \xi^\T A \xi$.
\par
In order to prove the bound $\lvert F_w^A \rvert \leq C_F$ we use the formulas $F_w^A = \psi (\sigma) F_\sigma^A - \sigma \psi' (\sigma)$ and $F_\sigma^{A_0} = d-2$ from Lemma~\ref{lemma:basic_properties}, our previous bound $\lvert F_\sigma^A - F_\sigma^{A_0} \rvert \leq \sigma C'_F$, $\psi' (\sigma) = \mu \psi (\sigma)$, $\psi (\sigma) \leq \euler^{\mu \sqrt{\elliptic_1}}$ and $\sigma \leq \sqrt{\elliptic_1}$ on $B_1$ to obtain
\begin{equation*}
\lvert F_w^A \rvert
\leq \psi (\sigma) \left( \left\lvert F_\sigma^A - F_\sigma^{A_0} \right\rvert + \left\lvert F_\sigma^{A_0} - \sigma \mu \right\rvert \right) 
 \leq C_F. 
\end{equation*}
\par
For our last bound $\lvert L_0 \psi (\sigma) \rvert \leq C_\psi / \sigma$ we note that $\nabla \psi (\sigma) = \mu  \euler^{\mu \sigma} \sigma^{-1} A_0^{-1} x$, use the product rule for the divergence and Eq.~\eqref{eq:BsigmaA} to obtain
\begin{align*}
- L_0\psi(\sigma)&=\diver(A\nabla\psi(\sigma)) =\mu\left(\nabla\left(\frac{\euler^{\mu\sigma}}{\sigma}\right)^\T AA_0^{-1}x+\frac{\euler^{\mu\sigma}}\sigma\diver(AA_0^{-1}x)\right)\\
&=\mu \euler^{\mu\sigma} \left(\frac{\mu}{\sigma^2}x^\T A_0^{-1}AA_0^{-1}x-
\frac{1}{\sigma^3}x^\T A_0^{-1}AA_0^{-1}x
+\frac{1}\sigma\diver(AA_0^{-1}x)\right)  \\
&=\mu \euler^{\mu\sigma} \frac{x^\T A_0^{-1}AA_0^{-1}x}{\sigma^3} \left(\mu \sigma
- 1
+ d + (F_\sigma^A - F_\sigma^{A_0}) \right) .
\end{align*}
The result follows from $x^\T A_0^{-1}AA_0^{-1}x \leq \elliptic_1^2 \sigma^2$, $\sigma \leq \sqrt{\elliptic_1}$ and $\lvert F_\sigma^A - F_\sigma^{A_0} \rvert \leq \sigma C'_F$.
\end{proof}
%
%
% %
% %%%%%%%%%%%%%%%%%%%%%%%%%%%%%%%%%%%%%%%%%%%%%%%%%%%%%%%%%%%%%%%%%%%%%%
% %---------------------------------------------------------------------
% 
%         P R O O F   O F   M A I N   T H E O R E M
%
% %---------------------------------------------------------------------
% %%%%%%%%%%%%%%%%%%%%%%%%%%%%%%%%%%%%%%%%%%%%%%%%%%%%%%%%%%%%%%%%%%%%%%
% %
%
%
\section{Proof of Theorem~\ref{thm:carleman}}\label{sec:proof}
As mentioned in the introduction, the proof of Theorem~\ref{thm:carleman} is based on ideas from \cite{EscauriazaV-03,BourgainK-05}. One might think that once the quantitative bounds from Section~\ref{sec:quantitative} are at hand, the result is obtained simply by following the existing proofs \cite{EscauriazaV-03,BourgainK-05}. However, due to the fact that we want to have an explicit dependence on all the parameters and also treat the case $\kappa = 1$ with additional gradient term in the lower bound at the same time (in the sense of Ineq.~\eqref{eq:EV03}), we need to employ non-trivial extensions of the techniques used in the mentioned papers. 
More precisely, the proof of Carleman estimate in \cite{EscauriazaV-03} has the gradient term in the lower bound which is sometimes needed in the applications, but the non-scaled version is valid on $B_\kappa \setminus \{0\}$ with undetermined $\kappa < 1$, and the proof does not allow for a scaling which keeps the constants explicit. The Carleman estimate in \cite{BourgainK-05}, on the other hand, is valid on the whole punctured unit ball in the non-scaled version, but the gradient cannot be easily preserved without breaking the proof. Our extension of their techniques addresses this issue. 
\par
First we prove a special case of the theorem and assume that
\begin{equation}\label{eq:case1}
u \in C_{\mathrm c}^\infty (B_1 \setminus \{0\}), \quad 
\rho = 1, \quad  
b , c \equiv 0, 
\quad \text{and} \quad
 u \ \text{is real-valued.}   
\end{equation}
Note that in this case we have $L = L_0$.
The general case then follows by regularization, scaling, the fact that Carleman estimates are stable under first order perturbations, and by adding the obtained Carleman estimate for $\Re u$ and $\Im u$. These steps are carried out in detail at the end of this section.
\par
We set $f = w^{-\alpha} u$. The following lemma can be found in \cite{EscauriazaV-03,MorassiRV-11}. For completeness we give a proof in Appendix~\ref{app:Lemma4.1}.
\begin{lemma}
\label{lemma_step1}
For all $\alpha > 0$ we have
\begin{multline} \label{eq:lemma1}
I_1 := \int \frac{w^2}{\nabla w^\T A\nabla w}(w^{-\alpha} L_0 u)^2 
\\
\geq 4\alpha\int (\nabla f)^\T M_w^A \nabla f - \alpha\int F_w^A L_0 (f^2)
+4\alpha^2\int\frac{\nabla w^\T A\nabla w}{w^2}D_f^2 .
\end{multline}
\end{lemma}
We start the proof by providing a lower bound on the first term of the right hand side of Ineq.~\eqref{eq:lemma1}.
With the notation
\[
 \tilde\nabla f : =   \nabla f -  \frac{\nabla \sigma \nabla \sigma^\T}{\nabla \sigma^\T A \nabla \sigma} A \nabla f = \nabla f - \frac{\nabla w \nabla w^\T}{\nabla w^\T A \nabla w} A \nabla f
\]
we have
\[
 \left( \frac{\nabla \sigma \nabla \sigma^\T}{\nabla \sigma^\T A \nabla \sigma} A \nabla f \right)^\T A \tilde\nabla f = 0 ,
 \]
and using the first identity of Lemma~\ref{lemma:basic_properties} we obtain
 \begin{align*}
  \nabla f^\T M_w^A \nabla f 
  &= \psi (\sigma) \nabla f^\T M_\sigma^A \nabla f 
  + \sigma \psi' (\sigma) \nabla f^\T A \tilde\nabla f \\
  &= \psi (\sigma) \nabla f^\T M_\sigma^A \nabla f 
  + \sigma \psi' (\sigma) 
  \left( \tilde\nabla f + \frac{\nabla \sigma \nabla \sigma^\T}{\nabla \sigma^\T A \nabla \sigma} A \nabla f \right)^\T
  A \tilde\nabla f  \\[1ex]
  &=\psi (\sigma) \nabla f^\T M_\sigma^A \nabla f 
  + \sigma \psi' (\sigma) 
  \tilde\nabla f^\T
  A \tilde\nabla f  .
 \end{align*}
Now using $M_\sigma^A \nabla \sigma = 0$ established in Lemma~\ref{lemma:basic_properties} and the fact that $M_\sigma^A$ is symmetric, we obtain 
\[
 \nabla f^\T M_w^A \nabla f =  \psi (\sigma) \tilde\nabla f^\T M_\sigma^A \tilde\nabla f 
  + \sigma \psi' (\sigma) \tilde\nabla f^\T A \tilde\nabla f .
\]
From Proposition~\ref{prop:quantitative} we obtain for almost all $x \in B_1$
\begin{equation}\label{eq:first_term}
 \nabla f^\T M_w^A \nabla f \geq  \left(\sigma \psi' (\sigma) - \psi (\sigma) \sigma C_M \right) \tilde\nabla f^\T A \tilde\nabla f .
\end{equation}
Using Ineq.~\eqref{eq:first_term} we provide a lower bound on the first two terms of the right hand side of Ineq.~\eqref{eq:lemma1}. Lemma~\ref{lemma:basic_properties} implies
\begin{equation} \label{eq:second_term}
 F_w^A = \psi(\sigma) (d-2) + \psi(\sigma) B_\sigma^A - \psi' (\sigma) \sigma  , \quad \text{where} \quad B_\sigma^A := F_\sigma^A - F_\sigma^{A_0} .
\end{equation}
We use Ineq.~\eqref{eq:first_term}, Eq.~\eqref{eq:second_term}, $L_0 (f^2) = 2 f L_0 f - 2 \nabla f^\T A \nabla f$,
\begin{equation}
\label{eq:f-tildef}
 \nabla f^\T A \nabla f = \tilde\nabla f^\T A \tilde\nabla f + \frac{\left(\nabla w^\T A \nabla f \right)^2}{\nabla w^\T A \nabla w} 
\end{equation}
and Green's theorem, i.e.\ $\int u L_0 v = \int \nabla u^\T A \nabla v$ for $u,v \in C_{\mathrm{c}}^2 (B_1)$, to obtain
\begin{multline*} 
I_2 := 4\alpha\int (\nabla f)^\T M_w^A \nabla f - \alpha\int F_w^A L_0 (f^2) 
 \\ \geq 4\alpha\int \left(\sigma \psi' (\sigma) - \psi (\sigma) \sigma C_M \right) \tilde\nabla f^\T A \tilde\nabla f
%\\
- (d-2) \alpha \int f^2 L_0 \psi (\sigma)  \\ - 2\alpha \int \left( \psi (\sigma) B_\sigma^A - \psi' (\sigma) \sigma \right) \left( fL_0 f - \tilde\nabla f^\T A \tilde\nabla f - \frac{(\nabla w^\T A \nabla f)^2}{\nabla w^\T A \nabla w} \right) .
\end{multline*}
We apply Eq.~\eqref{LwrtLaplace} to obtain
\begin{multline} \label{eq:I2}
I_2 \geq 
2\alpha\int \left(\sigma \psi' (\sigma) - 2\psi (\sigma) \sigma C_M + \psi (\sigma) B_\sigma^A \right) \tilde\nabla f^\T A \tilde\nabla f
\\
- 2 \alpha^3 \int \left(\psi (\sigma) B_\sigma^A - \psi' (\sigma) \sigma \right) \frac{\nabla w^\T A \nabla w}{w^2} f^2 - R_1 , 
\end{multline}
where
\begin{multline*}
 R_1 := (d-2) \alpha \int f^2 L_0 \psi (\sigma) + 4\alpha^2 \int  \left(\psi (\sigma) B_\sigma^A - \psi' (\sigma) \sigma \right) f D_f \frac{\nabla w^\T A \nabla w}{w^2}
 \\
 - 2\alpha \int  \left(\psi (\sigma) B_\sigma^A - \psi' (\sigma) \sigma \right)\left[ \frac{(\nabla w^\T A \nabla f)^2}{\nabla w^\T A \nabla w} - w^{-\alpha} f L_0 u  \right] .
\end{multline*}
From Lemma~\ref{lemma_step1} and Eq.~\eqref{eq:I2} we infer that
\begin{multline*}
 R_1 +\int \frac{w^2}{\nabla w^\T A\nabla w}(w^{-\alpha}L_0 u)^2 \geq 2\alpha\int \left(\sigma \psi' (\sigma) - 2\psi (\sigma) \sigma C_M + \psi (\sigma) B_\sigma^A \right) \tilde\nabla f^\T A \tilde\nabla f \\
- 2 \alpha^3 \int \left(\psi (\sigma) B_\sigma^A - \psi' (\sigma) \sigma \right) \frac{\nabla w^\T A \nabla w}{w^2} f^2 
 + 4\alpha^2\int\frac{\nabla w^\T A\nabla w}{w^2}D_f ^2.
\end{multline*}
We use the identity $\psi'=\mu \psi$ and the estimate $\lvert B_\sigma^A\rvert \leq C_M\sigma$ from Proposition~\ref{prop:quantitative} and obtain with the notation $C_\mu := \mu - 3 C_M$
\begin{multline} \label{eq:intermediate}
R_1 +\int \frac{w^2}{\nabla w^\T A\nabla w}(w^{-\alpha}L_0 u)^2
\geq 2\alpha C_\mu \int \sigma \psi(\sigma) \tilde\nabla f^\T A \tilde\nabla f \\
+ 2 \alpha^3 C_\mu \int \sigma\psi(\sigma) \frac{\nabla w^\T A \nabla w}{w^2} f^2 
+ 4\alpha^2\int\frac{\nabla w^\T A\nabla w}{w^2}D_f^2.
\end{multline}
Note that $\mu - B_\sigma^A / \sigma > \mu - 3C_M = C_\mu > 0$ by our assumption on $\mu$ and Proposition~\ref{prop:quantitative}. The identities \eqref{eq:f-tildef},
\begin{align*}
  \nabla f^\T A \nabla f 
 &= w^{-2\alpha} \nabla u^\T A \nabla u
 + \alpha^2 w^{- 2} f^2 \nabla w^\T A \nabla w 
 - 2 \alpha w^{-\alpha -1} f \nabla w^\T A \nabla u  \\[1ex]
 &= w^{-2\alpha} \nabla u^\T A \nabla u
 - \alpha^2 w^{- 2} f^2 \nabla w^\T A \nabla w 
 - 2 \alpha w^{-1} f \nabla w^\T A \nabla f  
\end{align*}
and
\[
f\nabla w^\T A\nabla f=f w^{-1}D_f \nabla w^\T A\nabla w 
-\frac 12f^2w^{-1}F_w^A\nabla w^\T A\nabla w 
\]
imply
\begin{multline*}
\tilde \nabla f^\T A \tilde \nabla f 
 = w^{-2\alpha} \nabla u^\T A \nabla u
 - \alpha^2 w^{- 2} f^2 \nabla w^\T A \nabla w 
 - 2 \alpha f w^{-2} D_f \nabla w^\T A\nabla w  \\
 + \alpha w^{-2} f^2 F_w^A \nabla w^\T A\nabla w 
 - \frac{(\nabla w^\T A \nabla f)^2}{\nabla w^\T A \nabla w}  .
\end{multline*}
Combining this with Ineq.~\eqref{eq:intermediate}, and using $\sigma \geq w$, $\psi \geq 1$ and the bound $F_w^A \geq -C_F$ from Proposition~\ref{prop:quantitative}, we derive at
 \begin{multline} \label{eq:intermediate2}
R_2
+\int \frac{w^2}{\nabla w^\T A\nabla w}(w^{-\alpha}L_0 u)^2
\geq 2 \alpha C_\mu \int w^{-2\alpha + 1} \nabla u^\T A \nabla u
\\
- 2C_\mu  C_F \alpha^2 \int\sigma \psi (\sigma)  \frac{\nabla w^\T A\nabla w}{w^2} f^2 
+4\alpha^2\int\frac{\nabla w^\T A\nabla w}{w^2}D_f^2 ,
\end{multline}
where 
\begin{multline*}
 R_2 := (d-2) \alpha \int f^2 L_0 \psi (\sigma) + 4\alpha^2 \int  \left(\psi (\sigma) B_\sigma^A - \psi' (\sigma) \sigma + C_\mu \sigma \psi (\sigma) \right) f D_f \frac{\nabla w^\T A \nabla w}{w^2}\\
- 2\alpha \int  \left(\psi (\sigma) B_\sigma^A - \psi' (\sigma) \sigma \right)\left[ \frac{(\nabla w^\T A \nabla f)^2}{\nabla w^\T A \nabla w} - w^{-\alpha} f L_0 u\right]
+ 2 C_\mu \alpha \! \int \! \sigma \psi (\sigma) \frac{(\nabla w^\T A \nabla f)^2}{\nabla w^\T A \nabla w}.
\end{multline*}
We drop the (positive) term $2\alpha C_\mu \int \sigma \psi(\sigma) \tilde\nabla f^\T A \tilde\nabla f$ in  Ineq.~\eqref{eq:intermediate} and add the inequality obtained in this way with Ineq.~\eqref{eq:intermediate2}. This gives us
 \begin{multline} \label{eq:quasiend1}
R
+ \int \frac{w^2}{\nabla w^\T A\nabla w}(w^{-\alpha}L_0 u)^2
\geq \alpha C_\mu \int w^{-2\alpha + 1} \nabla u^\T A \nabla u
\\
+ \alpha^2 (\alpha - C_F ) C_\mu  \int \sigma \psi (\sigma)  \frac{\nabla w^\T A\nabla w}{w^2} f^2
+4\alpha^2\int\frac{\nabla w^\T A\nabla w}{w^2}D_f^2  ,
\end{multline}
where $R := (R_1+R_2)/2$ is given by
\begin{multline*}
 R  
 = (d-2) \alpha \int f^2 L_0 \psi (\sigma) 
 + 4 \alpha^2 \int  \left(\psi (\sigma) B_\sigma^A - \psi' (\sigma) \sigma + \frac{C_\mu}{2} \sigma \psi(\sigma) \right) f D_f \frac{\nabla w^\T A \nabla w}{w^2}\\ 
- 2 \alpha \int  \left(\psi (\sigma) B_\sigma^A - \psi' (\sigma) \sigma - \frac{C_\mu}{2}  \sigma \psi (\sigma) \right) \frac{(\nabla w^\T A \nabla f)^2}{\nabla w^\T A \nabla w}  \\
+ 2 \alpha \int  \left(\psi (\sigma) B_\sigma^A - \psi' (\sigma) \sigma \right)   w^{-\alpha} f L_0 u. 
\end{multline*}
Next we provide an upper bound on $\lvert R \rvert$. Inequality $\lvert B_\sigma^A\rvert \leq C_M \sigma$ from Proposition~\ref{prop:quantitative} and our assumption $\mu > 3 C_M$ imply the estimates
\begin{align*}
 \left\lvert \frac{B_\sigma^A}{\sigma} - \mu + \frac{C_\mu}{2} \right\rvert
 &=
 \left\lvert \frac{B_\sigma^A}{\sigma} - \frac{\mu}{2} - \frac{3 C_M}{2} \right\rvert \leq \frac{4}{3} \mu ,\\
  \left\lvert \frac{B_\sigma^A}{\sigma} - \mu - \frac{C_\mu}{2} \right\rvert
  &= \left\lvert \frac{B_\sigma^A}{\sigma} - \frac{3}{2}\mu + \frac{3}{2} C_M \right\rvert \leq \frac{3}{2} \mu - \frac{1}{2} C_M \leq \frac{3}{2} \mu
\end{align*}
and $\lvert B_\sigma^A / \sigma - \mu \rvert \leq 4 \mu / 3$. Hence, using $\lvert L_0 \psi(\sigma)\rvert\leq C_\psi/\sigma$ from Proposition~\ref{prop:quantitative}, $w\leq \sigma$, $\psi'=\mu\psi$, $\sigma\leq \mu_1 w$, $\psi (\sigma) \leq \euler^{\mu\sqrt{\elliptic_1}}$ and $\nabla w^\T A\nabla w\leq (\varphi' (\sigma))^2 \elliptic_1^2 \leq \elliptic_1^2$ we obtain
\begin{align*}
 \left\lvert R \right\rvert
&\leq C_\psi d\alpha\int \frac{f^2}{w}
+ 3 \alpha \mu \mu_1 \euler^{\mu\sqrt{\elliptic_1}} 
\int \left(
2\alpha \elliptic_1^2    \frac{\lvert fD_f \rvert}{w}
+  w\frac{(\nabla w^\T A\nabla f)^2}{\nabla w^\T A\nabla w}
+  w^{1-\alpha}\lvert fL_0 u\rvert
\right) \\
& \leq K \left[
\alpha\int \frac{f^2}{w}
+\alpha^2\int \frac{\lvert fD_f\rvert}{w}
+\alpha \int w\frac{(\nabla w^\T A\nabla f)^2}{\nabla w^\T A\nabla w}
+\alpha\int w^{1-\alpha}\lvert fL_0 u\rvert 
\right] ,
\end{align*}
where
\[
 K :=\max\left\{d C_\psi, 6\mu \mu_1 \euler^{\mu  \sqrt{\elliptic_1}} \elliptic_1^2  \right\} 
 \leq 6 d \mu \mu_1 \euler^{\mu \sqrt{\elliptic_1}}\elliptic_1^2   \left( \sqrt{\elliptic_1} (C_F' + \mu) + d \right).
\] 
Using $(a-b)^2 \leq 2a^2 + 2b^2$, $\nabla w^\T A\nabla w\leq \elliptic_1^2$ and $w\leq \sigma\leq \sqrt{\elliptic_1}$ on $B_1$, we obtain
\begin{align*}
\frac {w}{\elliptic_1^2}\frac{(\nabla w^\T A\nabla f)^2}{\nabla w^\T A \nabla w}
\leq w\frac{(\nabla w^\T A\nabla f)^2}{(\nabla w^\T A \nabla w)^2}
\leq 2 \frac{D_f^2}{w} + \frac{C_F^2}{2} \frac{f^2}{w} 
\leq 2 \sqrt{\elliptic_1} \frac{D_f^2}{w^2} + \frac{C_F^2}{2} \frac{f^2}{w}.
\end{align*}
Furthermore, for all $t > 0$ we have 
\begin{align*}
\int f^2&\leq \sqrt{\elliptic_1} \int \frac{f^2}{w},\\
 \alpha\int w^{1-\alpha} \lvert fL_0 u\rvert &\leq \frac{\sqrt{\elliptic_1}\alpha^2}{2}\int\frac{f^2}{w}+ \frac{1}{2} \int w^{2-2\alpha}\lvert L_0 u\rvert ^2, \ \text{and}\\
 \int w^{-1}\lvert f D_f \rvert & \leq t \int w^{-1}f^2+\frac{\sqrt{\elliptic_1}}{4t}\int w^{-2}D_f^2 .
\end{align*}
Hence,
\begin{multline}\label{quasiend2}
\frac{\lvert R \rvert}{K} 
\leq 
 \alpha\left(1 +\alpha t + \frac{C_F^2\elliptic_1^2}{2} +\frac{\alpha \sqrt{\elliptic_1}}{2}\right)\int w^{-1}f^2  \\
 +\alpha \left(\alpha \frac{\sqrt{\elliptic_1}}{4t}+2 \elliptic_1^{5/2}\right)\int w^{-2}D_f^2
+ \frac{1}{2}\int w^{2-2\alpha}\lvert L_0 u\rvert ^2.
\end{multline}
From Ineq.~\eqref{eq:quasiend1}, $\sigma \psi (\sigma) \geq w$, the bound $\nabla w^\T A \nabla w \geq \elliptic_1^{-2} \mu_1^{-2} \euler^{-2\mu \sqrt{\elliptic_1}}$ and Ineq.~\eqref{quasiend2}, we finally obtain
\begin{equation} \label{eq:almost}
K_1 \int w^{2-2\alpha}\lvert L_0 u\rvert^2
\geq 
K_2 \int w^{-1}f^2
+K_3 \int w^{-2}D_f^2 
+K_4 \int w^{-2\alpha+1}\nabla u^\T A\nabla u ,
\end{equation}
where
\begin{align*}
 K_1 &:= K/2 +  (\elliptic_1 \mu_1)^2 \euler^{2\mu\sqrt{\elliptic_1}}, \\
 K_2 &:= C_\mu \frac{\euler^{-2\mu\sqrt{\elliptic_1}}}{(\elliptic_1 \mu_1)^2} \alpha^3 - \left[ C_\mu C_F \frac{\euler^{-2\mu\sqrt{\elliptic_1}}}{(\elliptic_1 \mu_1)^2} + K \left(t + \frac{\sqrt{\elliptic_1}}{2} \right) \right] \alpha^2 - K\left[ 1 + \frac{C_F^2 \elliptic_1^2}{2} \right] \alpha, \\
 K_3 &:= \left[ 4\frac{\euler^{-2\mu\sqrt{\elliptic_1}}}{(\elliptic_1\mu_1)^2} - K \frac{\sqrt{\elliptic_1}}{4t} \right]\alpha^2 -  2 K \elliptic_1^{5/2} \alpha \ \text{and} \\
 K_4 &:= C_\mu \alpha.
\end{align*}
Now we choose $t$ large enough, such that the coefficient of $\alpha^2$ in $K_3$ is positive, and thereafter we choose $\alpha$ sufficiently large, such that $K_3$ is non-negative. Our particular choice is
\[
 t = t_1 := \frac{1}{8} K  \elliptic_1^{5/2} \mu_1^2 \euler^{2\mu \sqrt{\elliptic_1}}
 \quad \text{and} \quad
 \alpha\geq K \elliptic_1^{9/2} \mu_1^2\euler^{2\mu\sqrt{\elliptic_1}} =:\alpha_1.
\]
Then
\[
 K_3 = 2 \frac{\euler^{-2\mu \sqrt{\elliptic_1}}}{(\elliptic_1 \mu_1)^2} \alpha^2-2K\elliptic_1^{5/2}\alpha \geq 0
\]
and Ineq.~\eqref{eq:almost} implies
\begin{equation} \label{eq:almostCarleman}
K_1 \int w^{2-2\alpha}\lvert L_0 u\rvert^2
\geq 
K_2 \int w^{-1}f^2
+K_4 \int w^{-2\alpha+1}\nabla u^\T A\nabla u .
\end{equation}
If additionally
\[
\alpha \geq  \frac{q}{p} + \sqrt{\frac{q^2}{p^2} + \frac{2r}{p}} =: \alpha_2, 
\]
where
\[
 p := C_\mu \frac{\euler^{-2\mu\sqrt{\elliptic_1}}}{(\elliptic_1\mu_1)^2},
 \quad
 q := C_\mu C_F \frac{\euler^{-2\mu\sqrt{\elliptic_1}}}{(\elliptic_1\mu_1)^2} + K \left(t_1 + \frac{\sqrt{\elliptic_1}}{2} \right),
 \quad 
 r := K\left( 1 + \frac{C_F^2 \elliptic_1^2}{2} \right) ,
\]
then
\[
 K_2\geq \frac 12 C_\mu \frac{\euler^{-2\mu\sqrt{\elliptic_1}}}{ (\elliptic_1\mu_1)^2}\alpha^3 =: K_5 \alpha^3.
\]
Moreover, it can be shown that $\alpha_1 \leq \alpha_2$. By using $\min \{K_5 , C_\mu\} = K_5$ we obtain from Ineq.~\eqref{eq:almostCarleman} for all $\alpha\geq \hat \alpha_0$
\begin{equation} \label{eq:Carleman_rho=1}
 \alpha^3\int w^{-1-2\alpha} u^2
+ \alpha \int w^{-2\alpha+1}\nabla u^\T A\nabla u 
\leq  \hat C \int w^{2-2\alpha} \bigl( L_0 u \bigr)^2 ,
\end{equation}
where
\begin{equation*}
 \hat \alpha_0 = \hat \alpha_0 (d , \elliptic_1 , \elliptic_2 , \mu) := \max\{\alpha_1 , \alpha_2\} = \alpha_2
 \quad \text{and} \quad
 \hat C =  \hat C (d , \elliptic_1 , \elliptic_2 , \mu) := \frac{K_1}{K_5} . %{\min \{K_5 , C_\mu\}}
\end{equation*}
This proves Theorem~\ref{thm:carleman} in case \eqref{eq:case1}.
\par
Now we lift the restriction that $u \in C_{\mathrm c}^\infty (B_1 \setminus \{0\})$ and assume that $u \in W^{2,2} (\RR^d)$ real-valued with support in $B_1 \setminus \{0\}$.
Let $\phi$ be a non-negative and real-valued function in $C_{\mathrm c}^\infty (\mathbb{R}^d)$ with the properties that $\lVert \phi \rVert_1 = 1$ and $\supp \phi \subset \overline{B_1}$. For $\epsilon > 0$ we define $\phi_\epsilon : \mathbb{R}^d \to \mathbb{R}_0^+$ by $\phi_\epsilon (x) = \epsilon^{-d} \phi (x/\epsilon)$. The function $\phi_\epsilon$ belongs to $C_{\mathrm c}^\infty (\mathbb{R}^d)$ and satisfies $\supp \phi_\epsilon \subset \overline{B_{\epsilon}}$. Now define
\[
 u_\epsilon = \phi_\epsilon \ast u, \quad u_\epsilon (x) = \int_{\mathbb{R}^d} \phi_\epsilon (x-y) u(y) \drm y .
\]
If $\epsilon$ is small enough, say $\epsilon \in (0 , \epsilon_0)$, then $u_\epsilon \in C_{\mathrm{c}}^\infty (B_1 \setminus \{0\})$, see \cite[Theorem~1.6.1]{Ziemer-89}. Hence we can apply Ineq.~\eqref{eq:Carleman_rho=1} to the function $u_\epsilon$.
By definition we have
\begin{align*}
 L_0 u_\epsilon &= -\sum_{i,j=1}^d \partial_i (a^{ij} \partial_j (\phi_\epsilon \ast u)) 
= -\sum_{i,j=1}^d \left(\partial_i a^{ij} \right)\left( \phi_\epsilon \ast \partial_j u \right) -\sum_{i,j=1}^d a^{ij} \left( \phi_\epsilon \ast \partial_i \partial_j u \right) .
\end{align*}
Since $\partial_j u$ and $\partial_i \partial_j u$ are elements of $L^2 (\mathbb{R}^d)$, \cite[Theorem~1.6.1 (iii)]{Ziemer-89} tells us that $\phi_\epsilon \ast \partial_j u \to \partial_j u$ and $\phi_\epsilon \ast \partial_i \partial_j u \to \partial_i \partial_j u$ in $L^2 (\mathbb{R}^d)$ as $\epsilon$ tends to zero. 
Hence $L_0 u_{\epsilon}\to L_0 u$ in $L^2$ as $\epsilon\to 0$.
Since obviously $\nabla u_{\epsilon}\to \nabla u$ and $u_{\epsilon}\to u$ in $L^2$ as $\epsilon\to 0$, 
and $w^{-1}$ is bounded both from below and above on $\supp u_\epsilon$ uniformly for all $\epsilon \in (0,\epsilon_0)$, we obtain Ineq.~\eqref{eq:Carleman_rho=1}
for $u \in W^{2,2} (\RR^d)$ real-valued with support in $B_1 \setminus \{0\}$ in the case $b,c \equiv 0$.
\par
Now we lift the restriction on $\rho$ and assume that $\rho > 0$ arbitrary and $u \in W^{2,2} (\RR^d)$ real-valued with support in $B_\rho \setminus \{0\}$.
We introduce the scaled coefficient functions $\tilde a^{ij} (x) := a^{ij} (\rho x)$ on $B_1$, and define the scaled elliptic operator 
\[
\tilde L_0 := - \sum_{i,j=1}^d \partial_i \bigl(\tilde a^{ij} \partial_j \bigr) .
\]
Obviously, the Lipschitz constant of $\tilde L_0$ is $\rho\elliptic_2$ and the ellipticity constant is $\elliptic_1$. By our assumption on $\mu$, we can apply Ineq.~\eqref{eq:Carleman_rho=1} to the function $\tilde u : B_1 \to \RR$, $\tilde u (x) := u (\rho x)$, and obtain for all $\alpha \geq  \tilde \alpha_0$ 
\begin{equation*} 
 \alpha \int   \tilde w^{1-2\alpha} \nabla \tilde u^\T A \nabla \tilde u  + \alpha^3 \int \tilde w^{-1-2\alpha} \tilde u^2    
 \leq \tilde C \int_{\mathbb{R}^d} \tilde w^{2-2\alpha} \bigl( \tilde L_0 \tilde u \bigr)^2,
\end{equation*}
where 
\begin{equation} \label{eq:tildeC+alpha}
\tilde \alpha_0 := \hat \alpha_0 (d , \elliptic_1 ,\rho \elliptic_2 , \mu), \quad   \tilde C :=  \hat C (d , \elliptic_1 ,\rho \elliptic_2 , \mu)
\end{equation}
and $\tilde w (x) := w (\rho x)$ for $x \in B_1$. Note $w = w_{\rho , \mu} = \varphi (\sigma(x / \rho))$ and hence $\tilde w (x) = w_{1,\mu} (x)$.
By the change of variables $y=x \rho$, we obtain for all $\alpha \geq \tilde\alpha_0$
\begin{equation} \label{eq:carleman2}
 \alpha \rho^2 \int w^{1-2\alpha} \nabla u^\T A \nabla u  + \alpha^3 \int w^{-1-2\alpha} u^2    
 \leq \tilde C \rho^4 \int_{\mathbb{R}^d} w^{2-2\alpha} \bigl( L_0 \tilde u \bigr)^2 ,
\end{equation}
which proves the theorem in the case $u \in W^{2,2} (\RR^d)$ real-valued with support in $B_\rho \setminus \{0\}$ and $b,c \equiv 0$.
\par
Now we lift the restriction that $u$ is real-valued and, i.e.\ we allow complex-valued functions. 
We apply Ineq.~\eqref{eq:carleman2} to the real part $\Re u$ and imaginary part $\Im u$ and add these two inequalities. 
We obtain Ineq.~\eqref{eq:carleman2} for $u \in W^{2,2} (\RR^d)$ complex-valued, since $\nabla$ and $L_0$ commute with $\Re$ and $\Im$, $A$ is positive, and $(\nabla \Re u)^\T A (\nabla \Re u) + (\nabla \Im u)^\T A (\nabla \Im u) = \nabla u^\T A \overline{\nabla u}$.
\par
Finally we lift the restriction on $b$ and $c$. Let $u \in W^{2,2} (\RR^d)$ be complex-valued with support in $B_\rho \setminus \{0\}$ and let $b, c \in L^\infty (B_\rho)$ be arbitrary complex valued functions.
From $|a_1 + a_2 + a_3|^2 \leq 3(|a_1|^2 + |a_2|^2 + |a_3|^2)$ for complex numbers $a_i$ we obtain
\[
 | L_0 u |^2 
 \leq 3 \left( | Lu |^2  +  | b^\T \nabla u |^2 +  | cu |^2 \right)
 \leq 3 \left ( | Lu |^2  + \lVert  b  \rVert_\infty^2 \elliptic_1 \nabla u^\T A \overline{\nabla u} + \lVert c \rVert_\infty^2 |u|^2 \right) .
 \] 
By Ineq.~\eqref{eq:carleman2} and $w (x) \leq \sqrt{\elliptic_1}$ on $B_\rho$, it follows that for all $\alpha \geq \tilde\alpha_0$ we have
\begin{multline*}
  \alpha \rho^2 \int w^{-2\alpha+1}\nabla u^\T A\overline{\nabla u} + \alpha^3\int w^{-1-2\alpha} |u|^2 
  \leq  
  3 \tilde C \rho^4 \int  w^{2-2\alpha} \left\lvert Lu \right\rvert^2   
  \\
+ 3 \tilde C \rho^4 \lVert  b  \rVert_{\infty}^2 \elliptic_1^{3/2} \int w^{1-2\alpha} \nabla u^\T A \overline{\nabla u} 
+ 3 \tilde C \rho^4 \lVert c \rVert_{\infty}^2 \elliptic_1^{3/2} \int  w^{-1-2\alpha}  |u|^2  .
\end{multline*}
Now we can subsume the two last terms on the right hand side into the left hand side by choosing $\alpha$ sufficiently large. In particular, setting
\begin{equation}
  \label{eq:alpha+C}
C := 6 \tilde C  
\quad\text{and}\quad  
\alpha_0 :=\max \left\{ \tilde\alpha_0, C \rho^2  \lVert  b  \rVert_{\infty}^2 \elliptic_1^{3/2},
C^{1/3} \rho^{4/3} \lVert c \rVert_{\infty}^{2/3}\sqrt{\elliptic_1}   \right\},
\end{equation} 
we obtain the statement of the theorem for all $\alpha \geq  \alpha_0$ and all $u \in W^{2,2} (\RR^d)$ with support in $B_\rho \setminus \{0\}$.\qed
%
% %
% %%%%%%%%%%%%%%%%%%%%%%%%%%%%%%%%%%%%%%%%%%%%%%%%%%%%%%%%%%%%%%%%%%%%%%
% %---------------------------------------------------------------------
% 
%         A P P E N D I X
%
% %---------------------------------------------------------------------
% %%%%%%%%%%%%%%%%%%%%%%%%%%%%%%%%%%%%%%%%%%%%%%%%%%%%%%%%%%%%%%%%%%%%%%
% %
%
%
\appendix

\section{Proof of Lemma~\ref{lemma:basic_properties}} \label{app:lemma_2.1}
First we prove the fourth relation. We use $\nabla w^\T A \nabla w = \sigma^{-2} \varphi' (\sigma)^2 x^\T A_0^{-1} A A_0^{-1} x$,
   \begin{equation} \label{eq:Lw}
    -L_0 w = \varphi'' (\sigma) \frac{1}{\sigma^2} x^\T A_0^{-1} A A_0^{-1} x - \varphi' (\sigma) \left[ \frac{x^\T A_0^{-1} A A_0^{-1} x}{\sigma^3} - \frac{\diver \left( A A_0^{-1} x\right)}{\sigma} \right]
   \end{equation}
   and since $\rho = 1$
   \begin{equation} \label{psi'}
    \sigma \psi' (\sigma) = 1 - \frac{w}{\sigma \varphi' (\sigma)} - \frac{w \varphi'' (\sigma)}{\varphi' (\sigma)^2},
   \end{equation}
   and obtain
   \begin{equation} \label{diver1}
    F_w^A =  -1 + \frac{w\varphi'' (\sigma) }{ \varphi' (\sigma)^2}
           - \frac{w}{\sigma \varphi' (\sigma)}
           + \frac{w \sigma \diver \left( A A_0^{-1} x \right)}{ \varphi' (\sigma) x^\T A_0^{-1} A A_0^{-1} x} . 
   \end{equation}
Note that $w = \varphi \circ \sigma$ and Eq.~\eqref{eq:Lw} with $\varphi$ replaced by the identity reads 
\begin{equation} \label{eq:Lsigma}
 -L_0 \sigma = -\sigma^{-3} x^\T A_0^{-1} A A_0^{-1} x + \sigma^{-1} \diver \left( A A_0^{-1} x\right).
\end{equation}
Using $\nabla \sigma^\T A \nabla \sigma = \sigma^{-2} x^\T A_0^{-1} A A_0^{-1} x$ and Eq.~\eqref{eq:Lsigma} we have by the definition of $F_\sigma^A$
\begin{equation}\label{diver2}
\diver (A A_0^{-1} x) = \frac{x^\T A_0^{-1} A A_0^{-1} x}{\sigma^2} (F_\sigma^A + 2) .
\end{equation}
From Eq.~\eqref{psi'}, Eq.~\eqref{diver1}, Eq.~\eqref{diver2} and the definition of $\psi$ we obtain $F_w^A = -\sigma \psi' (\sigma) + \psi (\sigma) F_\sigma^A$, i.e.\ the fourth estimate of Lemma~\ref{lemma:basic_properties}.
\par
The fifth identity of Lemma~\ref{lemma:basic_properties} $F_\sigma^{A_0} = d-2$ is a direct consequence of $\nabla \sigma^\T A \nabla \sigma = \sigma^{-2}x^\T A_0^{-1} A A_0^{-1} x$, $\sigma^2 = x^\T A_0^{-1} x$ and Eq.~\eqref{eq:Lsigma}.
\par
We now turn to the proof of the first relation. By definition we have $M_w^A = \tilde S_1 + \tilde S_2 + \tilde S_3 + \tilde S_3^\T$ where 
\[
 \tilde S_1 :=-\frac{1}{2} F_w^{A} A, \quad 
 \tilde S_2 := \frac{1}{2} \diver \left( h_w^A \circ A \right), 
 \quad \text{and} \quad 
 \tilde S_3 :=  -\frac{1}{2} A D (h_w^A) .
\]
For the first summand we have, using the fourth identity of Lemma~\ref{lemma:basic_properties},
\[
\tilde S_1 = -\frac{1}{2} \left( \psi (\sigma) F_\sigma^A - \sigma \psi' (\sigma) \right) A .
\]
For the second term we calculate using the chain rule and the product rule for the divergence
\begin{equation*}
 \tilde S_2 = \frac{1}{2} \diver \left( \frac{\varphi (\sigma)}{\sigma \varphi' (\sigma)} \frac{\sigma A \nabla \sigma}{\nabla \sigma^\T A \nabla \sigma}  \circ A \right) 
 =\frac{1}{2} \sigma \psi' (\sigma) A + \frac{1}{2} \psi (\sigma) \diver \left( \frac{\sigma A \nabla \sigma}{\nabla \sigma^\T A \nabla \sigma} \circ A \right) .
\end{equation*}
For the third summand we use the chain rule to see that $D (h_w^A) = D (\psi (\sigma) h_\sigma^A)$. Hence, by the product rule
\[
 \tilde S_3 = -\frac{1}{2} \psi (\sigma) A D (h_\sigma^A) -  \frac{1}{2} A \psi' (\sigma) (\nabla \sigma) (h_\sigma^A)^\T .
\]
Putting everything together we obtain that
\begin{align*}
 M_w^A &= \psi (\sigma) M_\sigma^A 
 + \sigma \psi' (\sigma) A 
 - \frac{\psi' (\sigma)}{2} \left( A (\nabla \sigma) (h_\sigma^A)^\T +  h_\sigma^A (\nabla \sigma)^\T A  \right) \\
 &= \psi (\sigma) M_\sigma^A 
 + \sigma \psi'(\sigma) \left[  A 
 - \frac{1}{2} \left( A \frac{\nabla \sigma \nabla \sigma^{\perp}}{\nabla\sigma^{\perp} A \nabla \sigma} A + A \frac{\nabla \sigma \nabla \sigma^{\perp}}{\nabla\sigma^{\perp} A \nabla \sigma} A \right) \right] .
\end{align*}
\par
To prove the second relation $M_\sigma^{A_0} = 0$ we have by $h_\sigma^{A_0} = x$ 
\begin{align*}
 \diver \left( x \circ A_0 \right) = d A_0 
 \quad \text{and} \quad
 D (h_\sigma^A) = I .
 \end{align*}
 Hence, using $F_\sigma^{A_0} = d-2$ we obtain
\begin{align*}
 M_\sigma^{A_0} &= -\frac{1}{2} F_\sigma^{A_0} A_0 + \frac{1}{2} \diver \left( h_\sigma^A \circ A_0 \right) - \frac{1}{2} A_0 D (h_\sigma^A) - \frac{1}{2} D (h_\sigma^A)^\T A_0  = 0.
 \end{align*}
\par
For the proof of the third relation we infer from the proof of Proposition~\ref{prop:quantitative} that $M_\sigma^A  = (M_\sigma^A - M_\sigma^{A_0} A_0^{-1} A) = (S_5 + S_3 + S_3^\T)$.
Moreover, we have that $2 S_5 = S_7+S_8+ 2S_9$ and $2 S_3 = -A (S_{10} + S_{11} + S_{12})$. We rearrange the terms, recall that $g = \sigma^2 / x^\T A_0^{-1} A A_0^{-1} x$ and $\partial_k A = (\partial_k a^{ij})_{i,j=1}^d$, set $v = (x^\T A_0^{-1} (\partial_k A) A_0^{-1} x)_{k=1}^d$, use $D (A A_0^{-1} x) = A_0^{-1} A  + C$, where 
\[
C = \Bigl(\sum_{k=1}^d (\partial_i a^{jk}) e_k^\T A_0^{-1} x \Bigr)_{i,j=1}^d, 
\]
and obtain that $M_\sigma^A = \sum_{i=1}^6 T_i$ with
\begin{align*}
T_1 &:= \frac{g}{2} \left[ (\nabla a^{ij})^\T A A_0^{-1} x \right]_{i,j=1}^d -\frac{g}{2} C^\T  A,  \\[1ex]
T_2 &:= \frac{g A}{2 x^\T A_0^{-1} A A_0^{-1} x} v x^\T A_0^{-1} A  -\frac{g}{2} A  C, \\[1ex]
T_3 &:= - \frac{g A}{x^\T A_0^{-1} A A_0^{-1} x}  
x^\T A_0^{-1} A A_0^{-1} A A_0^{-1} x + \frac{g}{x^\T A_0^{-1} A A_0^{-1} x} 
A A_0^{-1} x x^\T A_0^{-1} A A_0^{-1}  A,  \\[1ex]
T_{4} &:= \frac{g}{x^\T A_0^{-1} A A_0^{-1} x} 
  A  A_0^{-1} A A_0^{-1} x x^\T A_0^{-1} A   -g A A_0^{-1} A , \\[1ex]
T_5 &:= 2A - \frac{2}{x^\T A_0^{-1} A A_0^{-1} x} 
A A_0^{-1} x x^\T A_0^{-1} A , \\[1ex]
T_6 &:= \frac{g}{2 x^\T A_0^{-1} A A_0^{-1} x} 
A A_0^{-1} x v^\T A - \frac{g A}{2 x^\T A_0^{-1} A A_0^{-1} x} v^\T A A_0^{-1} x.
\end{align*}
If we multiply from the right with $\nabla \sigma = \sigma^{-1} A_0^{-1} x$, then one easily sees that $T_i = 0$, $i \in \{3,4,5,6\}$. It remains to show that $T_1 \nabla \sigma = T_2 \nabla \sigma = 0$. For this purpose, we calculate
\begin{align*}
\left( \left[ (\nabla a^{ik})^\T A A_0^{-1} x \right]_{i,k=1}^d A_0^{-1} x \right)_j  
%&=  \sum_k    (\nabla a^{jk})^\T A A_0^{-1} x e_k^\T A_0^{-1} x \\
&= \sum_k \sum_i \partial_i a^{jk} e_i^\T A A_0^{-1} x e_k^\T A_0^{-1} x \\
%&= \sum_i e_i^\T A A_0^{-1} x \sum_k \partial_i a^{jk} e_k^\T A_0^{-1} x \\
&= \sum_i e_i^\T A A_0^{-1} x C_{ij} 
%= \sum_i \left( C^\T \right)_{ji} (A A_0^{-1} x)_i \\
= (C^\T A A_0^{-1} x)_j, 
\end{align*}
hence $\left[ (\nabla a^{ij})^\T A A_0^{-1} x \right]_{i,j=1}^d A_0^{-1} x = C^\T A A_0^{-1} x$ and $T_1 \nabla \sigma = 0$.

For $T_2\nabla \sigma$ we calculate
\[
 (C A_0^{-1} x)_i = \sum_{j,k=1}^d e_j^\T A_0^{-1} x \partial_i a^{jk} e_k^\T A_0^{-1} x
 = x^\T A_0^{-1} (\partial_i A) A_0^{-1} x = v_i
\]
and obtain
\[ \pushQED{\qed} 
T_2 \nabla \sigma = \frac{g}{2 \sigma} Av  -\frac{g}{2\sigma } A  C A_0^{-1} x  = \frac{g}{2 \sigma} A(v - C A_0^{-1} x) = 0 . \qedhere \popQED
\]
\section{Proof of Lemma~\ref{lemma_step1}} \label{app:Lemma4.1}
We recall that $f = w^{-\alpha} u$ and calculate
\begin{align} 
-w^{-\alpha}L_0 u 
&=\alpha(\alpha-1)w^{-2} f \nabla w^\T A\nabla w+2\alpha w^{-1} \nabla f^\T A \nabla w - \alpha fw^{-1}L_0 w - L_0 f \nonumber \\[1ex]
&=-L_0 f +\alpha^2 f w^{-2} \nabla w^\T A \nabla w +2\alpha w^{-2}(\nabla w^\T A\nabla w) D_f .\label{LwrtLaplace}
\end{align}
Hence,
\begin{equation} \label{eq:I1}
I_1
=\int \left( 4\alpha^2\frac{\nabla w^\T A\nabla w}{w^2}D_f^2 - 4\alpha D_f L_0 f 
 + 4\alpha ^3 \frac{\nabla w^\T A\nabla w}{w^2}D_f f \right) + P ,
\end{equation}
where
\[
 P := \int \left(\alpha^2\frac{\sqrt{\nabla w^\T A\nabla w}}w f -\frac{w}{\sqrt{\nabla w^\T A\nabla w}}L_0 f \right)^2 \geq 0 .
\]
We use Green's formula, i.e.\ $\int u L_0 v= \int \nabla u^\T A \nabla v$ for functions $u,v \in C_0^2(B_1)$, and obtain for the third term in Eq.~\eqref{eq:I1}
\begin{align*}
\int \frac{\nabla w^\T A\nabla w}{w^2} D_f f&=\int \frac fw \nabla w^\T A\nabla f - \int \frac{f^2}{2w}L_0 w -\int \frac{f^2}{2w^2}\nabla w^\T A\nabla w\\
&=\int \frac fw \nabla w^\T A\nabla f -\int \frac{f^2}{2w^2}\nabla w^\T A\nabla w  - \int\nabla \left(\frac{f^2}{2w} \right)^\T A \nabla w .
\end{align*}
By the quotient rule, the last term equals to the sum of the two first terms. Hence,
\[
 \int \frac{\nabla w^\T A\nabla w}{w^2} D_f f = 0 .
\]
For the second term in Eq.~\eqref{eq:I1} we have
\begin{align*}
4\alpha\int D_f L_0 f=4\alpha\int w\frac{\nabla w^\T A\nabla f}{\nabla w^\T A\nabla w}L_0 f+2\alpha\int F_w^A fL_0 f .
\end{align*}
By Green's formula, we have
\begin{align*}
2\alpha\int F_w^AfL_0 f&= 2\alpha\int \nabla (F_w^A f)^\T A \nabla f=2\alpha\int f \nabla f^\T A \nabla F_w^A + 2\alpha \int F_w^A\nabla f^\T A \nabla f\\
&=\alpha\int (\nabla F_w^A)^\T A\nabla (f^2) + 2\alpha \int F_w^A\nabla f^\T A\nabla f \\
&=\alpha\int F_w^A L_0 (f^2) + 2 \alpha \int F_w^A\nabla f^\T A\nabla f .
\end{align*}
Hence, the second term in Eq.~\eqref{eq:I1} reads
\begin{align}
4\alpha\int D_f L_0 f &=4\alpha\int w\frac{\nabla w^\T A\nabla f}{\nabla w^\T A\nabla w} L_0 f + 2\alpha\int F_w^A\nabla f^\T A\nabla f+\alpha\int F_w^A L_0 (f^2). \label{eq:second}
\end{align}
From Rellich's identity \cite[Eq.~(5.2)]{Necas-12} we obtain
\begin{equation}\label{eq:rellich}
  \int (h_w^A)^\T \nabla f L_0 f = - \frac{1}{2} \int \nabla f^\T B \nabla f ,
\end{equation}
where
\[
h_w^A = \frac{w A \nabla w}{\nabla w^\T A \nabla w} \quad \text{and} \quad B := \diver (h_w^A A) - A D(h_w^A) - D(h_w^A)^\T A.
\]
From Eq.~\eqref{eq:second} and \eqref{eq:rellich} we obtain
\[
4\alpha\int D_f L_0 f
=-2\alpha \int \nabla f^\T B \nabla f +2\alpha\int F_w^A\nabla f^\T A\nabla f+\alpha\int F_w^A L_0 (f^2). 
\]
By definition of $M_w^A$ this gives
\begin{equation*} \pushQED{\qed} 
- 4\alpha\int D_f L_0 f=4\alpha\int (\nabla f)^\T M_w^A \nabla f - \alpha\int F_w^AL_0 (f^2). \qedhere \popQED
\end{equation*}
%
%
%
% %
% %%%%%%%%%%%%%%%%%%%%%%%%%%%%%%%%%%%%%%%%%%%%%%%%%%%%%%%%%%%%%%%%%%%%%%
% %---------------------------------------------------------------------
% 
%         A C K N O W L E D G E M E N T
%
% %---------------------------------------------------------------------
% %%%%%%%%%%%%%%%%%%%%%%%%%%%%%%%%%%%%%%%%%%%%%%%%%%%%%%%%%%%%%%%%%%%%%%
% %
%
%
\section*{Acknowledgment}
Large parts of this work was done during visits of C.R.\ and M.T.\ in Zagreb and I.N.\ in Chemnitz. This exchange was partially financially supported by the DAAD and the Croatian Ministry of Science, Education and Sports through the PPP-grant ``Scale-uniform controllability of partial differential equations''. 
I.N.\ was also partially supported by HRZZ project grant 9345.

The authors are grateful to Ivan Veseli\'c who introduced us to this research topic, and to Denis Borisov for stimulating discussions and comments received in response to an earlier version of this manuscript.
The authors also thank Sergio Vessella for valuable comments and for pointing to the paper \cite{MorassiRV-11}.

%
%
% %
% %%%%%%%%%%%%%%%%%%%%%%%%%%%%%%%%%%%%%%%%%%%%%%%%%%%%%%%%%%%%%%%%%%%%%%
% %---------------------------------------------------------------------
% 
%         R E F E R E N C E S
%
% %---------------------------------------------------------------------
% %%%%%%%%%%%%%%%%%%%%%%%%%%%%%%%%%%%%%%%%%%%%%%%%%%%%%%%%%%%%%%%%%%%%%%
% %
%
%
\newcommand{\etalchar}[1]{$^{#1}$}

\end{document}